\documentclass[12pt]{article}

\usepackage{amsmath} 
\usepackage{amsthm} 
\usepackage{thm-restate} 
\usepackage{amsfonts} 
\usepackage{amssymb} 
\usepackage{stmaryrd} 
\usepackage{mathtools} 

\usepackage{bm} 
\usepackage{bbm} 

\usepackage{array} 
\usepackage[inline]{enumitem} 
\usepackage[in]{fullpage} 

\usepackage{tikz} 
\usetikzlibrary{calc} 
\usetikzlibrary{cd} 
\usetikzlibrary{decorations.markings} 

\usepackage{pdflscape} 

\usepackage{subcaption} 

\usepackage[%
  pdftex,
  plainpages=false,
  pdfpagelabels,
  colorlinks,
  citecolor=black,
  linkcolor=black,
  urlcolor=black,
  filecolor=black,
  bookmarksopen=false
]{hyperref} 
\usepackage[all]{hypcap} 

\newtheoremstyle{thmstyle}
  {\medskipamount}
  {\smallskipamount}
  {\slshape}
  {0pt}
  {\bfseries}
  {.}
  { }
  {\thmname{#1}\thmnumber{ #2}{\normalfont\thmnote{ (#3)}}}

\newtheoremstyle{plainstyle}
  {\medskipamount}
  {\smallskipamount}
  {\rmfamily}
  {0pt}
  {\bfseries}
  {.}
  { }
  {\thmname{#1}\thmnumber{ #2}{\normalfont\thmnote{ (#3)}}}

\theoremstyle{thmstyle}
\newtheorem{theorem}{Theorem}[section]
\newtheorem{lemma}[theorem]{Lemma}
\newtheorem{corollary}[theorem]{Corollary}
\newtheorem{proposition}[theorem]{Proposition}

\newtheorem{claim}[theorem]{Claim}

\theoremstyle{plainstyle}
\newtheorem{definition}[theorem]{Definition}
\newtheorem{remark}[theorem]{Remark}

\newenvironment{proofof}[1]{\begin{proof}[Proof of #1.]}{\end{proof}}

\setlist[enumerate]{label={\roman*.}, ref={(\roman*)}}

\newcommand{\df}{\stackrel{\text{def}}{=}}
\newcommand{\place}{\mathord{-}}

\newcommand{\comp}{\mathbin{\circ}}
\newcommand{\rest}{\mathord{\vert}}

\newcommand{\function}[2]{\colon #1 \rightarrow #2}
\newcommand{\injection}[2]{\colon #1 \rightarrowtail #2}

\DeclareMathOperator{\End}{End}
\DeclareMathOperator{\Aut}{Aut}
\DeclareMathOperator{\Fix}{Fix}
\DeclareMathOperator{\GL}{GL}
\DeclareMathOperator{\im}{im}

\newcommand{\NN}{\mathbb{N}}

\newcommand{\RR}{\mathbb{R}}

\newcommand{\One}{\mathbbm{1}}

\newcommand{\cC}{\mathcal{C}}

\newcommand{\cF}{\mathcal{F}}

\newcommand{\cU}{\mathcal{U}}
\newcommand{\cV}{\mathcal{V}}

\newcommand{\fS}{\mathfrak{S}}

\def\Holder{H\"{o}lder}

\def\bigraphdraw#1#2#3{
  \begin{tikzpicture}
    \pgfmathsetmacro{\topy}{(\n+1)*\leftvertsep/2}
    \foreach \i in {1,...,\n}{
      \pgfmathsetmacro{\y}{\topy - \i * \leftvertsep}
      \coordinate (P\i) at (0,\y);
      \filldraw (P\i) circle (\ptsize);
      \node[left] at (P\i) {$\i$};
    }
    \pgfmathsetmacro{\topyright}{(#1+1)*\rightvertsep/2}
    \foreach[count=\i] \s in {#2}{
      \pgfmathsetmacro{\y}{\topyright - \i * \rightvertsep}
      \coordinate (R\i) at (\horsep, \y);
      \filldraw (R\i) circle (\ptsize);
      \node[right] at (R\i) {$\{\s\}$};
      \foreach \j in \s {
        \draw (P\j) -- (R\i);
      }
    }
    \pgfmathsetmacro{\x}{\horsep/2}
    \node[below] at (\x,\capy) {#3};
  \end{tikzpicture}
}

\title{Left-cut-percolation and induced-Sidorenko bigraphs}
\author{Leonardo N.~Coregliano\thanks{Institute for Advanced Study, {\tt lenacore@ias.edu}. This material is based upon work supported by the National Science Foundation, and by the IAS School of Mathematics.}}
\begin{document}
\maketitle

\begin{abstract}
  A Sidorenko bigraph is one whose density in a bigraphon $W$ is minimized precisely when $W$ is
  constant. Several techniques of the literature to prove the Sidorenko property consist of
  decomposing (typically in a tree decomposition) the bigraph into smaller building blocks with
  stronger properties. One prominent such technique is that of $N$-decompositions of Conlon--Lee,
  which uses weakly \Holder\ (or weakly norming) bigraphs as building blocks. In turn, to obtain
  weakly \Holder\ bigraphs, it is typical to use the chain of implications reflection bigraph
  $\implies$ cut-percolating bigraph $\implies$ weakly \Holder\ bigraph. In an earlier result by the
  author with Razborov, we provided a generalization of $N$-decompositions, called reflective tree
  decompositions, that uses much weaker building blocks, called induced-Sidorenko bigraphs, to also
  obtain Sidorenko bigraphs.

  In this paper, we show that ``left-sided'' versions of the concepts of reflection bigraph and
  cut-percolating bigraph yield a similar chain of implications: left-reflection bigraph $\implies$
  left-cut-percolating bigraph $\implies$ induced-Sidorenko bigraph. We also show that under mild
  hypotheses, the ``left-sided'' analogue of the weakly \Holder\ property (which is also obtained
  via a similar chain of implications) can be used to improve bounds on another result of
  Conlon--Lee that roughly says that bigraphs with enough vertices on the right side of each
  realized degree have the Sidorenko property.
\end{abstract}

\section{Introduction}

In~\cite{Sid91} (see also~\cite{Sid93}), Sidorenko conjectured that if $\Omega=(X,\mu)$ and $\Lambda=(Y,\nu)$
are probability spaces, $W\function{X\times Y}{\RR_+}$ is a bounded measurable function (a \emph{bigraphon}),
and $G=(V_1,V_2,E)$ is a bipartite graph with a given bipartition $(V_1,V_2)$ (a \emph{bigraph}), then
\begin{align}\label{eq:Sidorenkooriginal}
  t(G,W) & \geq t(\rho,W)^{e(G)},
\end{align}
where $\rho$ denotes the edge bigraph, $e(G)\df\lvert E(G)\rvert$ is the
number of edges of $G$ and
\begin{align}\label{eq:tGW}
  t(G,W)
  & \df
  \int_{X^{V_1}\times Y^{V_2}} \prod_{(v,w)\in E(G)} W(x_v,y_w)\ d(\mu\otimes\nu)(x,y)
\end{align}
is the non-induced (labeled) density of $G$ in $W$. Bigraphs $G$ that satisfy~\eqref{eq:Sidorenkooriginal} for
every $W$ are called \emph{Sidorenko bigraphs}. In fact, Sidorenko's Conjecture is often studied under the
further assumption that $W$ is symmetric (i.e., $\Omega=\Lambda$ and $W(x,y)=W(y,x)$ for every $x,y\in X$) and
let us point out right away that even though some of the literature results cited in this introduction were
proved under this further assumption, they all extend straightforwardly to the asymmetric setting.

Quite a few of the known results on Sidorenko's Conjecture concern deducing that a bigraph is Sidorenko if it
can be decomposed into small parts that are all Sidorenko bigraphs. An alternative way of viewing such results
is that each of them recursively defines a subclass of Sidorenko bigraphs containing some base bigraphs
(typically only the edge bigraph) and that is closed under some ``valid'' amalgamations. One is then
interested in the case when ``valid'' amalgamations allow for highly non-trivial bigraphs in the common set of
the amalgamation.

Arguably, one of the richest such recursively defined subclasses of Sidorenko bigraphs is provided by Szegedy
in~\cite{Sze15}. However, to properly describe which amalgamations are ``valid'' in Szegedy's framework, extra
information needs to be carried around about the bigraphs $G$: essentially a family of probability
distributions on homomorphisms from $G$ to all possible target bigraphons, and ``validness'' of an
amalgamation of $G_1$ and $G_2$ along a set $V$ is determined by compatibility of the marginals of the
distributions on $V$ and a relative entropy inequality. To make the result concrete, Szegedy then shows that
the second condition is satisfied if we instead recursively enforce the common part $G_1\rest_V\cong
G_2\rest_V$ to be a forest (and start from the edge bigraph).

In a different but similar flavor, Conlon--Kim--Lee--Lee~\cite{CKLL18a} showed that that the class of
Sidorenko graphs contains all \emph{strongly tree decomposable} graphs, that is, graphs $G$ containing a tree
decomposition $T = (V(T),E(T))$ such that
\begin{enumerate}
\item For each tree vertex $U\in V(T)$, $G\rest_U$ is a tree.
\item For each tree edge $\{U_1,U_2\}\in E(T)$, the intersection $U_1\cap U_2$ induces a forest
  $G\rest_{U_1\cap U_2}$.
\item For each tree edge $\{U_1,U_2\}\in E(T)$ there is an isomorphism $f$ between the minimal subtrees of
  $G\rest_{U_1}$ and $G\rest_{U_2}$ that contain $U_1\cap U_2$ such that $f$ fixes $U_1\cap U_2$ pointwisely.
\end{enumerate}
This result was also recursively extended to higher-order strong tree decomposable graphs in~\cite{CKLL18b},
but even then the common sets in the amalgamations are still required to be forests.

Moving away from the case when the common sets in the amalgamations are required to be forests,
Conlon--Lee~\cite{CL17} generalized the result above by connecting it to the theory of \emph{weakly norming
  bigraphs} (a.k.a.\ \emph{weakly \Holder\ bigraphs}) of Hatami~\cite{Hat10}, i.e., bigraphs $G$ such that
$W\mapsto t(G,\lvert W\rvert)^{1/e(G)}$ is a norm in the space of bounded measurable functions $X\times
Y\to\RR$ (up to a.e.\ equivalence). Namely, they showed that if $N$ is weakly norming, then the same result
holds for \emph{$N$-decomposable} graphs, which are graphs $G$ containing a tree decomposition $T$ such that
\begin{enumerate}
\item For each tree vertex $U\in V(T)$, $G\rest_U$ is isomorphic to $N$.
\item For each tree edge $\{U_1,U_2\}\in E(T)$ there is an isomorphism between $G\rest_{U_1}$ and
  $G\rest_{U_2}$ that fixes $U_1\cap U_2$ pointwisely.
\end{enumerate}
The class of weakly norming bigraphs is now considerably well understood: in~\cite{DGHRR18}, it is
shown that weakly norming bigraphs are precisely those that satisfy the step Sidorenko property
studied in~\cite{KMPW19} and implicitly in~\cite[\S14.2]{Lov12} and in~\cite{LS21} weakly norming
bigraphs are characterized as those $G$ such that the functional $t(G,\place)$ defined
by~\eqref{eq:tGW} for bounded measurable functions $W\function{X\times Y}{\RR}$ that can also take
negative values is convex (see also~\cite{LS22} for an equivalence with a complex-valued version of
the weak norming property).

Returning to $N$-decompositions, their main disadvantage that very few bigraphs are weakly norming:
Hatami himself showed~\cite[Theorem~2.10(ii)]{Hat10} that any weakly norming bigraph is necessarily
biregular after removing isolated vertices and Sidorenko~\cite{Sid20} showed that every weakly
norming bigraph is necessarily edge-transitive. Nevertheless, in the same paper, Conlon--Lee showed
that many non-trivial examples of weakly norming bigraphs can be obtained through a chain of
implications: reflection bigraph $\implies$ cut-percolating bigraph $\implies$ weakly norming
bigraph (see Sections~\ref{subsec:cutperc} and~\ref{subsec:reflection} below for formal definitions
of these concepts). Perhaps one of the most interesting examples of weakly norming bigraphs arising
from this chain of implications are the incidence bigraphs of the complete $k$-uniform hypergraphs
$K^{(k)}_n$ (see Figure~\ref{fig:incidencecomplete}).

\begin{figure}[htbp]
  \begin{center}
    \input{incidencecomplete}
  \end{center}
\end{figure}

In a recent work with Razborov~\cite[Theorem~3.5]{CR21}, we showed that the strong tree decompositions and
$N$-decompositions results can be both unified and generalized under the weaker notion of reflective tree
decompositions. The formal definition\footnote{For the reader's convenience, this definition and the
  associated result are included in Appendix~\ref{sec:reftree}, but these are not necessary for our results.}
of reflective tree decompositions is more technical than that of $N$-decompositions, but the result in
particular implies that all $N$-decomposable bigraphs are Sidorenko bigraph even when $N$ is only required to
be an \emph{induced-Sidorenko} bigraph, that is, when
\begin{align*}
  \frac{t(N,W)}{t(\rho,W)^{e(N)}} & \geq \frac{t(H,W)}{t(\rho,W)^{e(H)}}
\end{align*}
for every induced subgraph $H$ of $N$ and every bigraphon $W\function{\Omega\times\Lambda}{\RR_+}$ that is
\emph{biregular}, i.e., it satisfies
\begin{align*}
  \int_X W(x',y)\ d\mu(x') = \int_Y W(x,y')\ d\nu(y') & = t(\rho,W)
\end{align*}
for almost every $x\in X$ and almost every $y\in Y$ (every weakly norming bigraph is induced-Sidorenko, see
Remark~\ref{rmk:weakHolder->indSid}).

It was already observed in~\cite{CR21} that there are several induced-Sidorenko bigraphs that are
not weakly norming: the class of induced-Sidorenko bigraph is closed under amalgamations with trees
along a single vertex and the amalgamation of $k$ copies of the $4$-cycle along the same edge
(a.k.a.\ the \emph{$k$-book bigraph}, see Figure~\ref{fig:book}) is induced-Sidorenko.

\begin{figure}[htb]
  \begin{center}
    \begingroup

\def\pointsize{1 pt}
\def\vertsep{1}
\def\horzsep{1}
\def\angle{30}
\def\diststep{0.5}
\def\vertsize{1.5cm}
\def\uppervertsize{0.75cm}

\newcommand{\bookfig}[1]{%
  \begin{tikzpicture}
    \path[use as bounding box] (0,\uppervertsize) rectangle (0,-\vertsize);
    
    \coordinate (P) at (0,0);
    \coordinate (Q) at (0,-\vertsep);

    \pgfmathsetmacro{\secondangle}{180-\angle}

    \foreach \i in {0,...,#1}{%
      \pgfmathsetmacro{\dist}{\i * \diststep}
      \coordinate (A\i) at ($(P) + (\horzsep,0) + (\angle:\dist)$);
      \coordinate (B\i) at ($(Q) + (\horzsep,0) + (-\angle:\dist)$);
      \coordinate (C\i) at ($(P) + (-\horzsep,0) + (\secondangle:\dist)$);
      \coordinate (D\i) at ($(Q) + (-\horzsep,0) + (-\secondangle:\dist)$);
    }

    \filldraw (P) circle (\pointsize);
    \filldraw (Q) circle (\pointsize);

    \foreach \i in {0,...,#1}{%
      \filldraw (A\i) circle (\pointsize);
      \filldraw (B\i) circle (\pointsize);
      \filldraw (C\i) circle (\pointsize);
      \filldraw (D\i) circle (\pointsize);

      \draw (P) -- (A\i) -- (B\i) -- (Q) -- cycle;
      \draw (P) -- (C\i) -- (D\i) -- (Q) -- cycle;
    }
  \end{tikzpicture}
}

\begin{subfigure}[b]{0.3\textwidth}
  \begin{center}
    \bookfig{0}
    \caption*{$B_2$}
  \end{center}
\end{subfigure}
\quad
\begin{subfigure}[b]{0.3\textwidth}
  \begin{center}
    \bookfig{1}
    \caption*{$B_4$}
  \end{center}
\end{subfigure}
\quad
\begin{subfigure}[b]{0.3\textwidth}
  \begin{center}
    \bookfig{2}
    \caption*{$B_6$}
  \end{center}
\end{subfigure}
\caption{Book bigraphs.}
\label{fig:book}
\endgroup
  \end{center}
\end{figure}
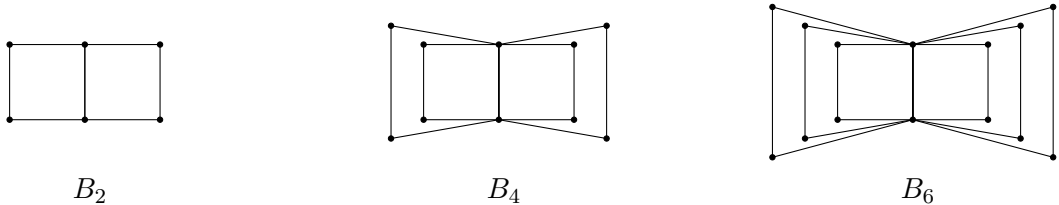

\medskip

Our first main result is to show (Theorems~\ref{thm:leftrefl->leftcutperc} and~\ref{thm:leftcutperc->indSid})
how ``left-sided'' versions of the concepts of reflection bigraphs and cut-percolation of Conlon--Lee yield a
similar chain of implications to obtain several examples of induced-Sidorenko bigraphs: left-reflection
bigraph $\implies$ left-cut-percolating bigraph $\implies$ induced-Sidorenko bigraph. Let us point out right
away that even though the natural analogue of weakly norming, called \emph{left-weakly \Holder}, also follows
from left-cut-percolation (Theorem~\ref{thm:leftcutperc->leftweakHolder}, see also
Theorem~\ref{thm:leftrefl->leftweakHolder}), it is a different notion from induced-Sidorenko (and our proof
requires the full power of left-cut-percolation to get induced-Sidorenko). See also
Figure~\ref{fig:implications} for a summary of the implications between these properties of bigraphs. The most
important example of left-reflection bigraph (Theorem~\ref{thm:incidencecomplete}) is the incidence bigraph of
the complete hypergraph $K_n^{k_1,\ldots,k_t}$ on $n$ vertices and in uniformities $k_1,\ldots,k_t$ (see
Figure~\ref{fig:incidencecomplete2}).

\begin{landscape}
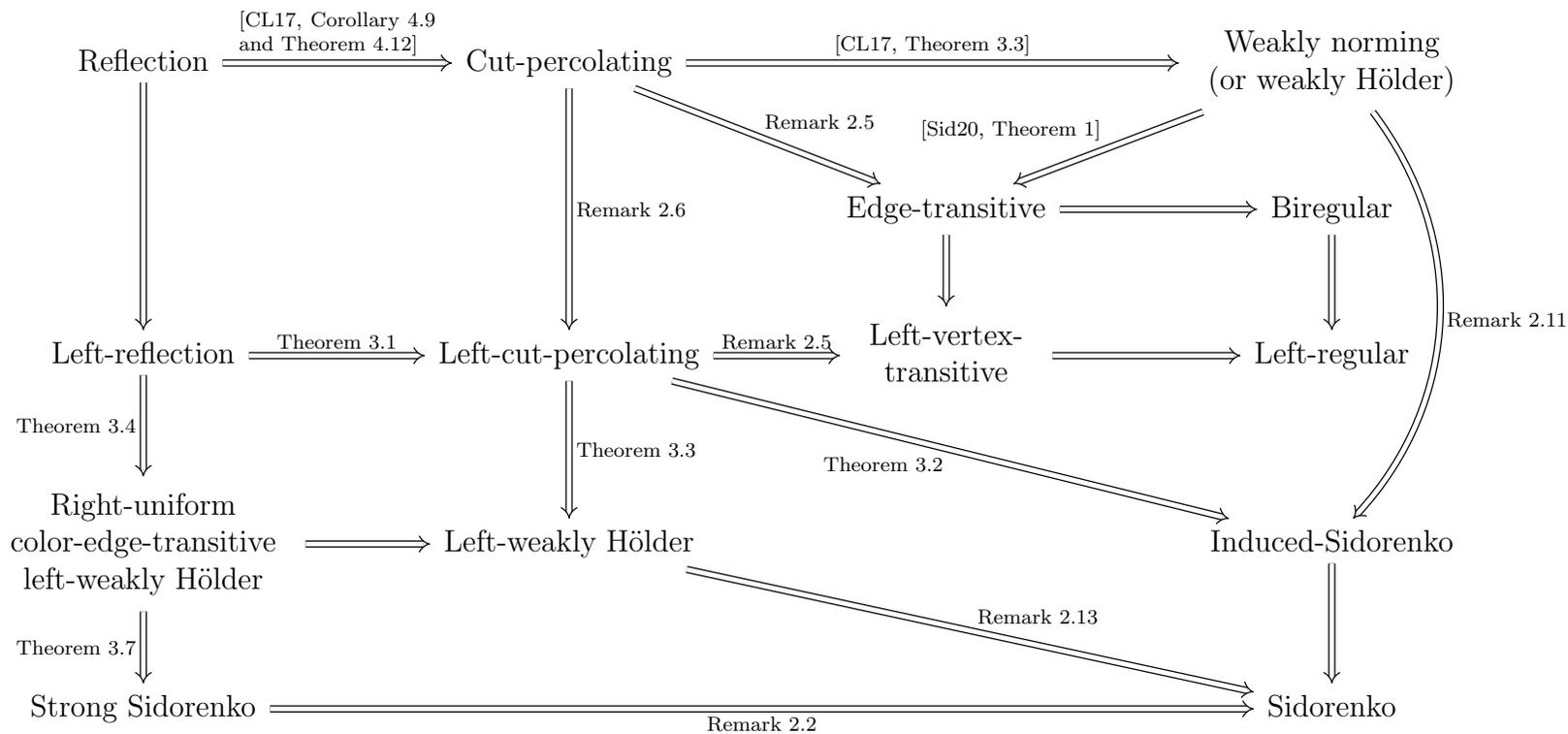
\begin{figure}[p]
  \begin{center}
    \begingroup

\begin{tikzcd}[%
    arrows={Rightarrow},
    row sep={1cm},
    column sep={1.65cm},
    math mode=false,
    every label/.append style = {font = \scriptsize},
  ]
  Reflection
  \arrow[r, "\protect{\parbox{2.7cm}{\cite[Corollary~4.9 and~Theorem~4.12]{CL17}}}"]
  \arrow[dd]
  &
  Cut-percolating
  \arrow[rr, "\protect{\cite[Theorem~3.3]{CL17}}"]
  \arrow[dd, "Remark~\ref{rmk:cutperc->leftcutperc}"]
  \arrow[dr, "Remark~\ref{rmk:transitivity}"]
  &
  &
  \begin{tabular}{c}
    Weakly norming\\
    (or weakly \Holder)
  \end{tabular}
  \arrow[dl, "\protect{\cite[Theorem~1]{Sid20}}"']
  \arrow[ddd, bend left=40, "Remark~\ref{rmk:weakHolder->indSid}"]
  \\
  &
  &
  Edge-transitive
  \arrow[r]\arrow[d]
  &
  Biregular
  \arrow[d]
  \\
  Left-reflection
  \arrow[r, "Theorem~\ref{thm:leftrefl->leftcutperc}"]
  \arrow[d, "Theorem~\ref{thm:leftrefl->leftweakHolder}"']
  &
  Left-cut-percolating
  \arrow[r, "Remark~\ref{rmk:transitivity}"]
  \arrow[d, "Theorem~\ref{thm:leftcutperc->leftweakHolder}"]
  \arrow[drr, "Theorem~\ref{thm:leftcutperc->indSid}"']
  &
  \begin{tabular}{c}
    Left-vertex-\\
    transitive
  \end{tabular}
  \arrow[r]
  &
  Left-regular
  \\
  \begin{tabular}{c}
    Right-uniform\\
    color-edge-transitive\\
    left-weakly \Holder
  \end{tabular}
  \arrow[r]
  \arrow[d, "Theorem~\ref{thm:leftweakHolderorbits}"']
  &
  Left-weakly \Holder
  \arrow[drr, "Remark~\ref{rmk:leftweakHolder->Sid}"]
  &
  &
  Induced-Sidorenko
  \arrow[d]
  \\
  Strong Sidorenko
  \arrow[rrr, "Remark~\ref{rmk:strongSid->Sid}"']
  &
  &
  &
  Sidorenko
\end{tikzcd}

\caption{Implications between properties considered under the assumption that the bigraph is
  non-trivial and does not have any isolated vertices. Arrows are labeled with the location of their
  proofs and arrows with easy proofs are either labeled by remarks or unlabeled (when the proof is
  trivial). In this diagram, the properties about colored bigraphs should be read as some coloring
  of the bigraph turns it into a colored bigraph with that property.}
\label{fig:implications}
\endgroup

  \end{center}
\end{figure}
\end{landscape}

\begin{figure}[htb]
  \begin{center}
    \input{incidencecomplete2}
  \end{center}
\end{figure}

\medskip

One of the most interesting applications of the theory of reflection bigraphs of Conlon--Lee (in fact, it uses
a hypergraph analogue of it) is the following theorem.
\begin{theorem}[Conlon--Lee~\protect{\cite[Theorem~1.1]{CL21}}]\label{thm:ConlonLee}
  Let $G=(V_1,V_2,E)$ be a bigraph, let $r$ be the maximum degree of a vertex in $V_2$ and for each
  $k\in\{2,\ldots,r\}$, let $d_k$ be the number of vertices in $V_2$ that have degree $k$.

  If $\binom{\lvert V_1\rvert}{r}\binom{r}{k}$ divides $d_k$ for every $k\in\{2,\ldots,r\}$, then $G$ is a
  Sidorenko bigraph.
\end{theorem}

As our second main result, we generalize the theorem above by weakening the divisibility condition to the
condition that $d_k$ is either zero or at least $\binom{\lvert V_1\rvert}{k}$ (Theorem~\ref{thm:largeright})
note that this also improves the first non-zero value of $d_k$ that is valid. In fact, we show that such
bigraphs satisfy a stronger inequality (see Definition~\ref{def:strongSidorenko}) that was initially used by
Sidorenko~\cite[Equation~(2)]{Sid91} to study his conjecture. This result is derived from the aforementioned
fact that the incidence bigraph of $K_n^{k_1,\ldots,k_t}$ is a left-reflection bigraph and we also show how
the left-weakly \Holder\ property (along with some extra mild properties) yields a similar result based on the
symmetries of the underlying bigraph (Theorem~\ref{thm:leftweakHolderorbits}).

\medskip

The paper is organized as follows. In Section~\ref{sec:prelim}, we give definitions and establish the notation
necessary to state our main results. In Section~\ref{sec:mainresults}, we state our main results. In
Section~\ref{sec:leftsided}, we prove all theorems that do not directly involve Sidorenko's Conjecture, that
is, Theorems~\ref{thm:leftrefl->leftcutperc}, \ref{thm:leftcutperc->leftweakHolder}
and~\ref{thm:incidencecomplete} (these theorems are mostly direct analogues of Conlon--Lee~\cite{CL21}). In
Section~\ref{sec:indSid}, we prove Theorem~\ref{thm:leftcutperc->indSid} that establishes the connection with
the induced-Sidorenko property. In Section~\ref{sec:orbits}, we prove Theorem~\ref{thm:largeright} on
Sidorenko bigraphs with enough vertices of each degree and its symmetries-based generalization,
Theorem~\ref{thm:leftweakHolderorbits}. We finish the paper with a brief discussion and some open
problems in Section~\ref{sec:conc}

\section{Preliminaries}
\label{sec:prelim}

Throughout the text, we will use the notation $\NN\df\{0,1,\ldots\}$ for non-negative integers and
$\NN_+\df\NN\setminus\{0\}$ for positive integers. For $n\in\NN$, we let $[n]\df\{1,\ldots,n\}$. We also let
$\RR$ be the set of real numbers and $\RR_+$ the set of non-negative real numbers. Given a set $V$, we denote
its power set by $2^V\df\{W\mid W\subseteq V\}$.

\subsection{Bigraphs}

A \emph{bigraph} is a triple $G=(V_1,V_2,E)$, where $V_1$ and $V_2$ are
disjoint finite sets and $E\subseteq V_1\times V_2$. We will also use the following notation ($i=1,2)$:
\begin{equation}\label{eq:bigraphshorthand}
  \begin{aligned}
    V_i(G) & \df V_i, &
    v_i(G) & \df \lvert V_i\rvert, &
    V(G) & \df V_1\cup V_2,
    \\
    E(G) & \df E, &
    e(G) & \df \lvert E\rvert, &
    v(G) & \df \lvert V_1\rvert + \lvert V_2\rvert.
  \end{aligned}
\end{equation}
For $v\in V(G)$, we denote its \emph{neighborhood} by
\begin{align*}
  N_G(v) & \df \{w\in V(G) \mid (v,w)\in E(G)\lor (w,v)\in E(G)\}.
\end{align*}
and its \emph{degree} by $d_G(v)\df\lvert N_G(v)\rvert$.

We say that $G$ is \emph{left $d$-regular} (\emph{right $d$-regular}, respectively) if $d_G(v) = d$ for every
$v\in V_1(G)$ ($v\in V_2(G)$, resp.). We say that $G$ is \emph{biregular} if it is both left $d_1$-regular and
right $d_2$-regular for some $d_1,d_2\in\NN$. An \emph{isomorphism} between bigraphs $G_1$ and $G_2$ is a
bijection $f\injection{V(G_1)}{V(G_2)}$ such that $f(V_i(G_1))=V_i(G_2)$ ($i=1,2$) and $(v,w)\in E(G_1)\iff
(f(v),f(w))\in E(G_2)$ ($(v,w)\in V_1(G_1)\times V_2(G_1)$); when such an isomorphism exists, we say that
$G_1$ and $G_2$ are \emph{isomorphic}, which is denoted $G_1\cong G_2$. An \emph{automorphism} of $G$ is an
isomorphism of $G$ to itself and we denote the set of automorphisms of $G$ by $\Aut(G)$. A \emph{homomorphism}
from a bigraph $G_1$ to a bigraph $G_2$ is a (not necessarily injective) map $f\function{V(G_1)}{V(G_2)}$ such
that $f(V_i(G_1))\subseteq f(V_i(G_2))$ ($i=1,2$) and $f(E(G_1))\subseteq E(G_2)$. An \emph{endomorphism} of
$G$ is a homomorphism of $G$ to itself. The set of endomorphisms of $G$ is denoted $\End(G)$.

For $U\subseteq V(G)$, we let $G\rest_U$ be the \emph{subgraph induced by $U$ in $G$}, that is, we let
\begin{align*}
  V_i(G\rest_U) & \df V_i(G)\cap U, &
  E(G\rest_U) & \df E(G)\cap ((U\cap V_1(G))\times (U\cap V_2(G))).
\end{align*}
We also let $G - U\df G\rest_{V(G)\setminus U}$. Furthermore, for a set of edges $E\subseteq E(G)$, we let $G
- E\df (V_1(G), V_2(G), E(G)\setminus E)$ be the subgraph obtained from $G$ by removing the edges in $E$.

We denote the \emph{edge bigraph} by $\rho\df(\{1\},\{2\},\{(1,2)\})$, the \emph{$d$-star bigraph} by
$K_{1,d}\df(\{0\},[d],\{(0,i)\mid i\in[d]\})$ and the \emph{dual $d$-star bigraph} by
$K_{d,1}\df([d],\{0\},\{(i,0)\mid i\in[d]\})$.

\subsection{Flags}
\label{subsec:flags}

It will be convenient to also work with partially labeled bigraphs and for this purpose we will borrow some
terminology from the theory of flag algebras~\cite{Raz07}.

More specifically, we work in the theory $T_{\operatorname{Graph}}^2$ of
graphs augmented with a 2-coloring of its vertices. Thus, a \emph{flag} is a
partially labeled bigraph, that is, a pair $F=(G,\theta)$, where $G$
is a bigraph and $\theta\injection{[k]}{V(G)}$ is an injection for some
$k\in\NN$. We use the notation $\lvert F\rvert\df G$ for the \emph{underlying bigraph}
of $F$ and the notation $\theta_F\df\theta$ for the \emph{labeling} of $F$.
We will often abuse notation and write
$F=(G,(\theta(1),\theta(2),\ldots,\theta(k)))$, listing the values of
$\theta$. In fact, we will abuse the notation even more and write $F=(G,U)$
for some set $U\subseteq V(G)$ to be understood as $F=(G,\theta)$ for some
$\theta\injection{[\lvert U\rvert]}{V(G)}$ with $\im(\theta)=U$, whenever the
exact ordering is either clear from the context or unimportant.

An \emph{isomorphism} between flags $F_1=(G_1,\theta_1)$ and $F_2=(G_2,\theta_2)$ is an isomorphism $f$
between $G_1$ and $G_2$ that preserves the partial labeling in the sense that $f\comp\theta_1=\theta_2$; when
such an isomorphism exists, we say that $F_1$ and $F_2$ are \emph{isomorphic} and denote it by $F_1\cong F_2$.

If $F_1=(G_1,\theta_1)$ and $F_2=(G_2,\theta_2)$ are flags such that $\theta_2\comp\theta_1^{-1}$ is an
isomorphism between $G_1\rest_{\im(\theta_1)}$ and $G_2\rest_{\im(\theta_2)}$ (that is, in the terminology of
flag algebras, these flags are of the same type), we let $F_1\sqcup F_2$ be the flag obtained from the
disjoint union of $F_1$ and $F_2$ by identifying vertices with the same label\footnote{We avoid using $F_1F_2$
  here to not conflict with the product as defined in flag algebras.}. For $i\in[2]$, we let
$e_i\df(\rho,i)$. For a bigraph $G$, we let $G^L\df(G,V_1(G))$ be the flag in which all left vertices of $G$
are labeled.

\subsection{Bigraphons}

Given probability spaces $\Omega=(X,\mu)$ and $\Lambda=(Y,\nu)$, a
\emph{bigraphon} over $\Omega$ and $\Lambda$ is a bounded measurable function
$W\function{X\times Y}{\RR_+}$, where $X\times Y$ is equipped with the
product $\sigma$-algebra and the product measure $\mu\otimes\nu$; we will
denote bigraphons by $W\function{\Omega\times\Lambda}{\RR_+}$.

When taking integrals, our functions will always be bounded and hence Fubini's Theorem will apply and we will
be omitting references to it. If $V$ is a set, we let $\Omega^V=(X^V,\mu^V)$ be the product probability space
of $\lvert V\rvert$ copies of $\Omega$; we will usually abuse notation and denote $\mu^V$ simply by
$\mu$. Given $x\in X^V$ and $S\subseteq V$, we let $x_S\in X^S$ be the projection of $x$ to the coordinates in
$S$.

For a bigraph $G$ and a bigraphon $W\function{\Omega\times\Lambda}{\RR_+}$,
we let $t(G,W)\in \RR_+$ be given by~\eqref{eq:tGW}. More generally, for a
flag $F=(G,\theta)$ and a bigraphon $W\function{\Omega\times\Lambda}{\RR_+}$,
we let the function
$t(F,W)\function{\Omega^{V_1(G)\cap\im(\theta)}\times\Lambda^{V_2(G)\cap\im(\theta)}}{\RR_+}$
be given by
\begin{align*}
  t(F,W)(x,y)
  & \df
  \int_{X^{V_1(G)\setminus\im(\theta)}\times Y^{V_2(G)\setminus\im(\theta)}}
  \prod_{(v,w)\in E(G)} W(x''_v,y''_w)
  \ d(\mu\otimes\nu)(x',y'),
\end{align*}
where
\begin{align*}
  x''_v
  & \df
  \begin{dcases*}
    x_v, & if $v\in V_1(G)\cap\im(\theta)$,\\
    x'_v, & if $v\in V_1(G)\setminus\im(\theta)$;
  \end{dcases*}
  &
  y''_w
  & \df
  \begin{dcases*}
    y_w, & if $w\in V_2(G)\cap\im(\theta)$,\\
    y'_w, & if $w\in V_2(G)\setminus\im(\theta)$.
  \end{dcases*}
\end{align*}
When $V_1(G)\cap\im(\theta) = \varnothing$, we will simplify the notation to
$t(F,W)(y)$, and likewise for $V_2(G)\cap\im(\theta) = \varnothing$.

A bigraphon $W\function{\Omega\times\Lambda}{\RR_+}$ is called \emph{left-regular} if it satisfies
\begin{align*}
  t(e_1,W)(x) & = t(\rho,W)
\end{align*}
for almost every $x\in X$. Dually, it is called \emph{right-regular} if
\begin{align*}
  t(e_2,W)(y) & = t(\rho,W)
\end{align*}
for almost every $y\in Y$. Finally, it is called \emph{biregular} if it is both left regular and right regular.

As mentioned in the introduction, a \emph{Sidorenko bigraph} $G$ is a bigraph such that $t(G,W)\geq
t(\rho,W)^{e(G)}$ for every bigraphon $W$. While studying Sidorenko bigraphs, Sidorenko considered a stronger
inequality~\cite[Equation~(2)]{Sid91} that yields the class of strong Sidorenko bigraphs defined below.

\begin{definition}\label{def:strongSidorenko}
  A bigraph $G$ is a \emph{strong Sidorenko bigraph}\footnote{In fact, Sidorenko's condition
    in~\cite[Equation~(2)]{Sid91} also involves global weight functions $f$ and $g$ and allows for arbitrary
    measure spaces (that are not necessarily probability spaces), but it is not hard to see that the condition
    stated here is equivalent (see Appendix~\ref{sec:strongSid}).} if for every bigraphon
  $W\function{\Omega\times\Lambda}{\RR_+}$ and all sequences $f=(f_v)_{v\in V_1(G)}$ and $g=(g_w)_{w\in
    V_2(G)}$ of bounded measurable functions $f_v\function{\Omega}{\RR_+}$, $g_w\function{\Lambda}{\RR_+}$, we
  have
  \begin{align*}
    t(G;f,g;W) & \geq t\left(\rho; \prod_{v\in V_1(G)} f_v^{1/e(G)}, \prod_{w\in V_2(G)} g_w^{1/e(G)}; W\right)^{e(G)},
  \end{align*}
  where
  \begin{align*}
    t(G;f,g;W)
    & \df
    \int_{X^{V_1(G)}\times Y^{V_2(G)}}
    \prod_{v\in V_1(G)} f_v(x_v)\cdot
    \prod_{w\in V_2(G)} g_w(y_w)\cdot
    \prod_{(v,w)\in E(G)} W(x_v,y_w)
    \ d(\mu\otimes\nu)(x,y)
  \end{align*}
  and
  \begin{multline*}
    t\left(\rho; \prod_{v\in V_1(G)} f_v^{1/e(G)}, \prod_{w\in V_2(G)} g_w^{1/e(G)}; W\right)
    \\
    \df
    \int_{X\times Y}
    \prod_{v\in V_1(G)} f_v(x)^{1/e(G)}\cdot
    \prod_{w\in V_2(G)} g_w(y)^{1/e(G)}\cdot
    W(x,y)
    \ d(\mu\otimes\nu)(x,y).
  \end{multline*}
\end{definition}

\begin{remark}\label{rmk:strongSid->Sid}
  By simply taking all functions $f_v$ and $g_w$ to be constant equal to $1$, we deduce that strong
  Sidorenko bigraphs are Sidorenko bigraphs. However, it is not hard to see that taking $W$ and
  $g_w$ to be constant equal to $1$, we retrieve an instance of \Holder's inequality for the $f_v$
  functions that only holds if $e(G)\geq v_1(G)$. Other than examples that satisfy $e(G) <
  \min\{v_1(G),v_2(G)\}$ (or are indirectly obtained from such examples), it is not known whether
  the strong Sidorenko property is strictly stronger than the Sidorenko property.
\end{remark}

We say that a bigraph $G_1$ \emph{weakly dominates} $G_2$ if
\begin{align*}
  \frac{t(G_1,W)}{t(\rho,W)^{e(G_1)}} & \geq \frac{t(G_2,W)}{t(\rho,W)^{e(G_2)}}
\end{align*}
for every biregular non-zero bigraphon $W$.

An \emph{induced-Sidorenko} bigraph $G$ is a bigraph that weakly dominates all of its induced subgraphs.

\subsection{Cut-percolation}
\label{subsec:cutperc}

We start by recalling the definitions of~\cite[\S 3]{CL17} pertaining cut-percolation.

A \emph{cut-involution} of a bigraph $G$ is an automorphism $\phi\in\Aut(G)$ that satisfies:
\begin{enumerate}
\item $\phi$ is an involution, i.e., $\phi = \phi^{-1}$.
\item The set $\Fix(\phi)$ of points that are fixed by $\phi$ is a vertex-cut in the bigraph $G$, i.e., $G -
  \Fix(G)$ is disconnected.
\end{enumerate}
(If $G$ is already disconnected, then the empty set is declared to be a vertex-cut.)

The subgroup of $\Aut(G)$ generated by cut-involutions of $G$ is called the \emph{cut-involution group of $G$}.

A \emph{fold}\footnote{Let us remark that in~\cite[\S 3]{CL17}, Conlon--Lee use the same name
  ``cut-involution'' for folds, leaving the choice of the set $L$ implicit.} of $G$ is a pair $(\phi,L)$,
where $\phi$ is a cut-involution of $G$ and $L\subseteq V(G)$ is such that
\begin{enumerate}
\item $G\rest_L$ is a union of connected components of $G - \Fix(\phi)$.
\item $(L,\Fix(\phi),\phi(L))$ is a partition of $V(G)$.
\end{enumerate}
The set $L$ is called the \emph{left side} of the fold $(\phi,L)$ (but note that it is not necessarily
contained in $V_1(G)$).

\begin{remark}\label{rmk:cutinv}
  Not every cut-involution can be completed to a fold. A simple counter-example is the bigraph $G$ of
  Figure~\ref{fig:cutinv} given by
  \begin{align*}
    V_1(G) & \df \{0,2,4\}, &
    V_2(G) & \df \{1,3,5,6\}, &
    E(G) & \df \{(0,1),(0,3),(0,5),(0,6),(2,1),(2,3),(4,1),(4,3)\}
  \end{align*}
  and the cut-involution $\phi$ that maps $(1,2,5)$ to $(3,4,6)$ and fixes $0$. In fact, it is straightforward
  to check that a necessary and sufficient condition for the existence of some $L$ so that $(\phi,L)$ is a
  fold is that no connected component of $G - \Fix(\phi)$ is fixed by $\phi$ as a set.
\end{remark}

\begin{figure}[htbp]
  \begin{center}
    \begingroup

\def\cycleh{2}
\def\cyclesteph{0.75}
\def\cyclew{1.5}
\def\lowh{-1}
\def\loww{1}
\def\ptsize{2pt}

\begin{tikzpicture}
  \coordinate (P0) at (0,0);
  \coordinate (P1) at (-\cyclew,\cycleh);
  \coordinate (P2) at ($(0,\cycleh) + (0,\cyclesteph)$);
  \coordinate (P3) at (\cyclew,\cycleh);
  \coordinate (P4) at ($(0,\cycleh) + (0,-\cyclesteph)$);
  \coordinate (P5) at (-\loww,\lowh);
  \coordinate (P6) at (\loww,\lowh);

  \foreach \i in {0,...,6}{
    \filldraw (P\i) circle (\ptsize);
  }

  \node[below] at (P0) {$0$};
  \node[left] at (P1) {$1$};
  \node[above] at (P2) {$2$};
  \node[right] at (P3) {$3$};
  \node[below] at (P4) {$4$};
  \node[left] at (P5) {$5$};
  \node[right] at (P6) {$6$};

  \draw (P1) -- (P2) -- (P3) -- (P4) -- cycle;
  \draw (P1) -- (P0) -- (P3);
  \draw (P5) -- (P0) -- (P6);
\end{tikzpicture}

\caption{Example of Remark~\ref{rmk:cutinv} of a bigraph with a cut-involution that does not yield a fold.}
\label{fig:cutinv}

\endgroup

  \end{center}
\end{figure}
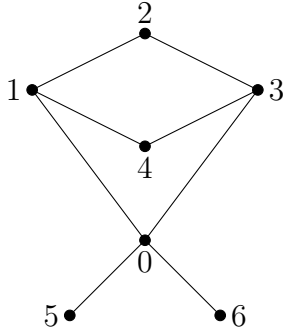

Given a fold $(\phi,L)$, the \emph{left-folding} and \emph{right-folding} maps are the maps
$\phi_L,\phi_L^*\function{V(G)}{V(G)}$, respectively, defined by
\begin{align*}
  \phi_L(v)
  & \df
  \begin{dcases*}
    \phi(v), & if $v\notin L$;\\
    v, & otherwise
  \end{dcases*}
  &
  \phi_L^*(v)
  & \df
  \begin{dcases*}
    \phi(v), & if $v\in L$;\\
    v, & otherwise.
  \end{dcases*}
\end{align*}
Note that these are endomorphisms of $G$.

A \emph{cut-percolating sequence} of a bigraph is a sequence of sets $E_0,E_1,\ldots,E_m\subseteq E(G)$ such
that $\lvert E_0\rvert = 1$, $E_m = E(G)$ and for every $i\in [m]$, there exists a fold $(\phi_i,L_i)$ of $G$
such that $E_i = (\phi_i)_{L_i}^{-1}(E_{i-1})$. In this definition, when want to make explicit the folds used
in the cut-percolating sequence, we say that it is a cut-percolating sequence with respect to $\Phi =
((\phi_1,L_1),\ldots,(\phi_m,L_m))$.

A bigraph is called \emph{cut-percolating} if it has a cut-percolating sequence. In fact, if $S$ is a set of
folds of $G$, we say that $G$ is \emph{cut-percolating under $S$} if it has a cut-percolating sequence with
respect to a sequence of folds in $S$.

The left-sided analogue of the above is defined in terms of left vertices.
\begin{definition}
  A \emph{left-cut-percolating sequence} of a bigraph is a sequence of sets $U_0,U_1,\ldots,U_m\subseteq
  V_1(G)$ such that $\lvert U_0\rvert = 1$, $U_m = V_1(G)$ and for every $i\in [m]$, there exists a fold
  $(\phi,L)$ of $G$ such that $U_i = \phi_L^{-1}(U_{i-1})$. Again, we say that this sequence is a
  left-cut-percolating sequence with respect to $\Phi=((\phi_1,L_1),\ldots,(\phi_m,L_m))$ when we want to make
  the sequence of folds explicit.

  A bigraph is called \emph{left-cut-percolating} if it has a left-cut-percolating sequence. For a set of
  folds $S$ of $G$, we say that $G$ \emph{left-cut-percolating under $S$} if it has a left-cut-percolating
  sequence with respect to a sequence of folds in $S$.
\end{definition}

\begin{remark}\label{rmk:transitivity}
  The cut-percolating sequence of a cut-percolating bigraph $G$ shows that the orbit of the single edge in
  $E_0$ under the action of the cut-involution group of $G$ is $E(G)$, so $G$ is edge-transitive under its
  cut-involution group action. By the same token, a left-cut-percolating bigraph is necessarily
  left-vertex-transitive under its cut-involution group action.

  In fact, both these statements trivially remain true replacing the cut-involution group of $G$ with its
  subgroup generated by the cut-involutions that appear in the sequence of folds used by the cut-percolating
  sequence.
\end{remark}

\begin{remark}\label{rmk:cutperc->leftcutperc}
  It is easy to see that every cut-percolating bigraph without isolated vertices is left-cut-percolating by
  simply tracking down the left endpoints of the edges in a cut-percolating sequence.
\end{remark}

\subsection{Reflection bigraphs}
\label{subsec:reflection}

Again, we start by recalling the definitions of~\cite[\S 4]{CL17} pertaining reflection bigraphs.

Let $R\subseteq\GL_n(\RR)$ be a finite reflection group (i.e., a finite group generated by reflections) and
let $T\subseteq R$ be the set of reflections in $R$. The set $\Phi$ of unit vectors that are orthogonal to
some hyperplane that is fixed by an element of $T$ is called \emph{root system of $R$} and its elements are
called \emph{roots}. Given further an ordered basis $\cU = (u_1,\ldots,u_n)$ of $\RR^n$, a root
$\alpha\in\Phi$ is called \emph{positive} (with respect to $\cU$) if it can be written as $\alpha =
\sum_{i=1}^n c_i u_i$ with $c_{i_0} > 0$, where $i_0\in[n]$ is the first index for which $c_i$ is non-zero;
otherwise, the root is called \emph{negative}. The set $\Phi$ is then partitioned into the sets $\Phi^+$ and
$\Phi^-$ of positive and negative roots, respectively.

For $\alpha\in\Phi$, let $H_\alpha$ be the hyperplane (through the origin) orthogonal to $\alpha$,
let $s_\alpha\in T$ be the reflection on $H_\alpha$ and let
\begin{align*}
  D_\alpha^+ & \df \{x\in\RR^n \mid \langle x,\alpha\rangle > 0\}; &
  D_\alpha^- & \df \{x\in\RR^n \mid \langle x,\alpha\rangle < 0\}.
\end{align*}
Dually, for $t\in T$, we let $\alpha_t\in\Phi^+$ be the unique positive root such that $s_{\alpha_t} = t$. We
also use the shorthands $H_t\df H_{\alpha_t}$, $D_t^+\df D_{\alpha_t}^+$ and $D_t^-\df D_{\alpha_t}^-$.

A set $\Delta\subseteq\Phi^+$ is a \emph{simple system} if every $\alpha\in\Phi^+$ can be written as a conic
combination (i.e., a linear combination with non-negative coefficients) of elements of $\Delta$ and $\Delta$
is minimal with this property. It is known (see~\cite[\S 1.3 and~1.5]{Hum90}) that $\Delta$ is unique with
respect to $\cU$, its elements are linearly independent and $S_\Delta\df\{s_\alpha \mid \alpha\in\Delta\}$
generates $R$. The elements of $\Delta$ are called \emph{simple roots} and the elements of $S_\Delta$ are
called \emph{simple reflections}.

For $I\subseteq S_\Delta$, let
\begin{align*}
  C(I) & \df \left(\bigcap_{s\in I} H_s\right)\cap\left(\bigcap_{s\in S_\Delta\setminus I} D_s^+\right).
\end{align*}

Given $S_1,S_2\subseteq S_\Delta$, the \emph{$(S_1,S_2;S_\Delta)$-reflection bigraph}\footnote{In~\cite{CL17},
  the group $R$ is also included in this notation, but since it can be retrieved from $S_\Delta$, we drop it from the
  notation here.} is the bigraph $G$ defined by
\begin{align*}
  V_1(G) & \df R/R_1 = \{rR_1 \mid r\in R\},\\
  V_2(G) & \df R/R_2 = \{rR_2 \mid r\in R\},\\
  E(G) & \df \{(rR_1, rR_2) \mid r\in R\},
\end{align*}
where $R_i$ is the subgroup of $R$ generated by $S_i$. In words, the left and right vertices are the left
cosets of $R_1$ and $R_2$, respectively and an edge is present whenever these left cosets intersect.

In~\cite[Corollary~4.9]{CL17}, Conlon--Lee showed that every reflection $t\in T$ naturally defines a fold
$(\phi(t;S_1,S_2;S_\Delta),L(t;S_1,S_2;S_\Delta))$ of the $(S_1,S_2;S_\Delta)$-reflection bigraph $G$ given by
\begin{gather*}
  \phi(t;S_1,S_2;S_\Delta)(r R_i) \df tr R_i, \qquad (r\in R, i\in[2])\\
  L(t;S_1,S_2;S_\Delta) \df \{rR_i \mid r(C(S_i))\subseteq D_{\alpha_t}^+, i\in [2]\}.
\end{gather*}
We will call such folds \emph{reflection folds} of the reflection bigraph $G$.

The left-sided analogue of reflection bigraphs uses multiple subgroups for the right vertices.
\begin{definition}
  Given $S_0,S_1,\ldots,S_k\subseteq S_\Delta$, the $(S_0;S_1,\ldots,S_k;S_\Delta)$-left-reflection bigraph is
  the amalgamation of the $(S_0,S_i;S_\Delta)$-reflection bigraphs $G_i$ ($i\in[k]$) over the left side, that is, it
  is obtained from their disjoint union by identifying the corresponding left vertices. Formally, it is the
  bigraph $G$ defined by
  \begin{align*}
    V_1(G) & \df R/R_0 = \{rR_0 \mid r\in R\},\\
    V_2(G) & \df \bigsqcup_{i=1}^k R/R_i \df \{(rR_i,i) \mid r\in R, i\in [k]\},\\
    E(G) & \df \{(rR_0, (rR_i,i)) \mid r\in R, i\in[k]\}.
  \end{align*}
  Equivalently, in flag language, we have $G\df \lvert\bigsqcup_{i=1}^k G_i^L\rvert$.

  For each $t\in T$, the \emph{reflection fold}
  $(\phi(t;S_0;S_1,\ldots,S_k;S_\Delta),L(t;S_0;S_1,\ldots,S_k;S_\Delta))$ of $G$ is the amalgamation of the
  reflection folds $(\phi(t;S_0,S_i;S_\Delta),L(t;S_0,S_i;S_\Delta))$ of the $G_i$; formally, we define
  \begin{align*}
    \phi(t;S_0;S_1,\ldots,S_k;S_\Delta)(r R_0) & \df tr R_0,\\
    \phi(t;S_0;S_1,\ldots,S_k;S_\Delta)((r R_i,i)) & \df (tr R_i,i),\\
    L(t;S_0;S_1,\ldots,S_k;S_\Delta)
    & \df
    \{rR_0 \mid r(C(S_0))\subseteq D_{\alpha_t}^+\}\cup\bigcup_{i\in[k]}
    \{(rR_i,i) \mid r(C(S_i))\subseteq D_{\alpha_t}^+\}.
  \end{align*}
  
  The \emph{natural coloring} of the $(S_0;S_1,\ldots,S_k;S_{\Delta})$-left-reflection-bigraph $G$ is the
  function $c\function{E(G)}{[k]}$ given by
  \begin{align*}
    c(rR_0, (rR_i,i)) & \df i.
  \end{align*}
\end{definition}

\subsection{Colored bigraphs and left-weakly \Holder\ bigraphs}

A \emph{colored bigraph} is a pair $H=(G,c)$, where $G$ is a bigraph and $c\function{E(G)}{C}$ is a function
called \emph{coloring}. We use the shorthand notations $c_H\df c$, $C_H\df C$ and $G(H)\df G$. We also abuse
notation of~\eqref{eq:bigraphshorthand} by applying it directly to $H$ (e.g., $E(H)\df E(G(H))$). For each
$i\in C$, we let $e_i(H)\df\lvert c^{-1}(i)\rvert$ be the number of edges that have color $i$ and for each
$v\in V(G)$, we let
\begin{align*}
  d_{H,i}(v) & \df \lvert\{w\in V(G) \mid (v,w)\in c^{-1}(i)\lor (w,v)\in c^{-1}(i)\}\rvert
\end{align*}
denote its \emph{$i$-degree in $H$}. We say that $H$ is \emph{left-color-regular} if $d_{H,i}(v) = d_{H,i}(w)$
for every $i\in C_H$ and every $v,w\in V_1(H)$ (which is equivalent to saying that $d_{H,i}(v) =
e_i(H)/v_1(H)$ for every $v\in V_1(H)$).

For $U\subseteq V(H)$, we let $H\rest_U\df(G(H)\rest_U,c_H\rest_{E(G(H)\rest_U)})$ be the \emph{colored
  bigraph induced by $U$}. For a set of colors $C'\subseteq C_H$, we let $H_{C'}=(G_{C'},c_{H_{C'}})$ be the
colored bigraph obtained from $H$ by keeping only edges that have color in $C'$, that is, we have $G_{C'}\df
G(H) - c_H^{-1}(C_H\setminus C')$ and $c_{H_{C'}}\df c_H\rest_{E(G_{C'})}$. We say that $H$ is
\emph{right-uniform} if the coloring $c$ factors as $c(v,w) = f(w)$ for some function
$f\function{V_2(G)}{C}$. An \emph{isomorphism} between colored bigraphs $H_1=(G_1,c_1)$ and $H_2=(G_2,c_2)$ is an
isomorphism $f$ between $G_1$ and $G_2$ that preserves the coloring in the sense that $c_1(v,w) =
c_2(f(v),f(w))$ for every $(v,w)\in E(G_1)$.  An \emph{automorphism} of $H$ is an isomorphism of $H$ to itself
and we denote the set of automorphisms of $H$ by $\Aut(H)$. A colored bigraph $H$ is called
\emph{color-edge-transitive} if for every $(v_1,w_1),(v_2,w_2)\in E(H)$ with $c_H(v_1,w_1)=c_H(v_2,w_2)$,
there exists an automorphism $\sigma\in\Aut(H)$ such that $\phi(v_1) = v_2$ and $\sigma(w_1) = w_2$.

Given a colored bigraph $H$ and a sequence of bigraphons $W = (W_i)_{i\in C_H}$ all over the same spaces
$\Omega=(X,\mu)$ and $\Lambda=(Y,\nu)$, we define
\begin{align*}
  t(H,W) & \df \int_{X^{V_1(H)}\times Y^{V_2(H)}} \prod_{(v,w)\in E(H)} W_{c_H(v,w)}(x_v,y_w)\ d(\mu\otimes\nu)(x,y).
\end{align*}

\begin{remark}\label{rmk:Autcolor}
  Note that if $\sigma\in\Aut(G)$, we have $t((G,c),W) = t((G,c\comp\sigma),W)$ (even if $\sigma$ is not in
  $\Aut((G,c))$) by simply renaming the integration variables.
\end{remark}

The notions of flags $F = (H,\theta)$ over colored bigraphs $H=(G,c)$ and the corresponding function $t(F,W)$
are defined analogously to Section~\ref{subsec:flags}.

As mentioned in the introduction, a \emph{weakly norming bigraph}, or a \emph{weakly \Holder\ bigraph} is a
bigraph $G$ such that $W\mapsto t(G,\lvert W\rvert)^{1/e(G)}$ defines a norm in the space of bounded functions
$X\times Y\to\RR$ up to a.e.\ equivalence. In~\cite[Theorem~2.10(ii)]{Hat10}, Hatami showed that a bigraph $G$
is weakly norming if and only if for every coloring $c\function{E(G)}{C}$ and for every sequence of bigraphons
$W = (W_i)_{i\in C}$, we have the following \Holder-like inequality
\begin{align*}
  t((G,c),W) & \leq \prod_{(v,w)\in E(G)} t(G,W_{c(v,w)})^{1/e(G)}.
\end{align*}
Furthermore, in~\cite{DGHRR18}, weakly norming bigraphs are further characterized as precisely those
bigraphs that have the step Sidorenko property studied in~\cite{KMPW19} and implicitly
in~\cite[\S14.2]{Lov12}.

The left-sided analogue of this is a bit more technical and is defined for colored bigraphs instead.
\begin{definition}
  A \emph{left-coloring} of a bigraph $G$ is a function $\ell\function{V_1(G)}{C}$.

  Given both a coloring $c\function{E(G)}{C'}$ and a left-coloring $\ell\function{V_1(G)}{C}$ of the same
  bigraph $G$, we define the coloring $\ell\otimes c\function{E(G)}{C\times C'}$ by
  \begin{align*}
    (\ell\otimes c)(v,w) & = (\ell(v),c(v,w)).
  \end{align*}

  A colored bigraph $H=(G,c)$ is called \emph{left-weakly \Holder} if for every left-coloring
  $\ell\function{V_1(G)}{C}$ of $G(H)$ and every sequence of bigraphons $W = (W_i)_{i\in C\times C_H}$, we
  have
  \begin{align*}
    t((G,\ell\otimes c),W) & \leq \prod_{v\in V_1(G)} t((G,\ell(v)\otimes c),W)^{1/v_1(G)},
  \end{align*}
  where $\ell(v)$ on the right-hand side is interpreted as the left-coloring $V_1(G)\to C$ that is constant
  equal to $\ell(v)$.
\end{definition}

\begin{remark}\label{rmk:leftweakHoldersubcolor}
  If $C'\subseteq C_H$ for a left-weakly \Holder\ bigraph $H$, then $H_{C'}$ is also left-weakly \Holder. This
  can be seen by setting $W_i\equiv 1$ for every $i\in C\times (C_H\setminus C')$ in the left-weak
  \Holder\ property of $H$.
\end{remark}

\begin{remark}\label{rmk:weakHolder->indSid}
  By~\cite[Theorem~2.14]{Hat10}, every weakly norming bigraph $G$ satisfies $t(G,W)^{1/e(G)}\geq
  t(H,W)^{1/e(H)}$ for any (not necessarily induced) subgraph $H$ of $G$, which in particular implies that $G$
  is induced-Sidorenko, since (assuming $G$ is non-empty) for an induced subgraph $H$ of $G$, we have
  \begin{align*}
    t(G,W)
    & =
    t(G,W)^{e(H)/e(G)}\cdot t(G,W)^{1-e(H)/e(G)}
    \geq
    t(H,W)\cdot t(\rho,W)^{e(G)-e(H)}.
  \end{align*}
\end{remark}

\begin{remark}\label{rmk:weakHolder->leftweakHolder}
  Since by~\cite[Theorem~2.10(ii)]{Hat10} every weakly norming bigraph without isolated vertices is biregular,
  it follows that if $G$ is weakly norming, without any isolated vertices and $c$ is a constant coloring of
  $G$, then $(G,c)$ is left-weakly \Holder, as for every left-coloring $\ell$, we have
  \begin{align*}
    t((G,\ell\otimes c),W)
    & \leq
    \prod_{(v,w)\in E(G)} t(G,W_{(\ell(v),c_0)})^{1/e(G)}
    =
    \prod_{v\in V_1(G)} t(G,W_{(\ell(v),c_0)})^{d_G(v)/e(G)}
    \\
    & =
    \prod_{v\in V_1(G)} t(G,W_{(\ell(v),c_0)})^{1/v_1(G)},
  \end{align*}
  where $c_0$ is the unique element in $\im(c)$.
\end{remark}

\begin{remark}\label{rmk:leftweakHolder->Sid}
  An argument analogous to that of Remark~\ref{rmk:weakHolder->leftweakHolder} can be used to show that the
  underlying bigraph of left-weakly \Holder\ bigraphs $H=(G,c)$ is Sidorenko: let $v_0\in V_1(G)$ be a
  non-isolated vertex, let $\ell\function{V_1(G)}{\{0,1\}}$ be given by $\ell(v)\df \One[v=v_0]$ and for a
  bigraphon $W$, considering the sequence $W'=(W'_{t,i})_{t\in\{0,1\},i\in C_H}$ given by $W'_{0,i}\df 1$ and
  $W'_{1,i}\df W$, we get
  \begin{align*}
    t(\rho,W)^{e(G)}
    & \leq
    t(K_{1,d_G(v_0)},W)^{e(G)/d_G(v_0)}
    =
    t((G,\ell\otimes c),W')^{e(G)/d_G(v_0)}
    \\
    & \leq
    t((G,\ell(v_0)\otimes c),W')^{e(G)/(v_1(G) d_G(v_0))}
    \prod_{v\in V_1(G)\setminus\{v_0\}}
    t((G,\ell(v)\otimes c),W')^{e(G)/(v_1(G) d_G(v_0))}
    \\
    & =
    t(G,W)^{e(G)/(v_1(G) d_G(v_0))}
    =
    t(G,W),
  \end{align*}
  where the first inequality follows from Jensen's Inequality and the last equality follows from
  Lemma~\ref{lem:obstacle} in Section~\ref{sec:indSid} below (as the derivation above holds for every
  bigraphon $W$). In fact, a similar argument in Lemma~\ref{lem:leftweakHolder->colorSid} will show that every
  left-weakly \Holder\ bigraph without isolated vertices is necessarily left-color-regular.
\end{remark}

\section{Main results}
\label{sec:mainresults}

In this section we state our main results.

\begin{theorem}\label{thm:leftrefl->leftcutperc}
  Every left-reflection bigraph is left-cut-percolating under reflection folds.
\end{theorem}

\begin{theorem}\label{thm:leftcutperc->indSid}
  Every left-cut-percolating bigraph is induced-Sidorenko.
\end{theorem}

\begin{restatable}{theorem}{leftcutperctoleftweakHolder}\label{thm:leftcutperc->leftweakHolder}
  Let $G$ be a left-cut-percolating bigraph under a set $S$ of folds of $G$. If $c\function{E(G)}{C}$ is a
  coloring of $G$ that is invariant under the subgroup of $\Aut(G)$ generated by $\{\phi \mid (\phi,L)\in
  S\}$, then $(G,c)$ is left-weakly \Holder.
\end{restatable}

\begin{theorem}\label{thm:leftrefl->leftweakHolder}
  Let $G$ be a left-reflection bigraph and let $c$ be its natural coloring. Then $(G,c)$ is right-uniform,
  color-edge-transitive and left-weakly \Holder.
\end{theorem}

\begin{theorem}\label{thm:incidencecomplete}
  Let $n,t\in\NN_+$ and let $k_1,\ldots,k_t\in[n]$. Then the incidence bigraph $G$ of the complete hypergraph
  on $n$ vertices and in uniformities $k_1,\ldots,k_t$ defined by
  \begin{align*}
    V_1(G) & \df [n],\\
    V_2(G) & \df \bigsqcup_{i=1}^t \binom{[n]}{k_i}
    \df \left\{(U,i) \;\middle\vert\; U\in\binom{[n]}{k_i}, i\in [t]\right\},\\
    E(G) & \df \left\{(v,(U,i)) \;\middle\vert\; v\in U, i\in [t], U\in\binom{[n]}{k_i}\right\}
  \end{align*}
  is a left-reflection bigraph.
\end{theorem}

As mentioned in the introduction, the next theorem generalizes Theorem~\ref{thm:ConlonLee}
from~\cite[Theorem~1.1]{CL21}.

\begin{restatable}{theorem}{largeright}\label{thm:largeright}
  Let $G$ be a bigraph without isolated vertices and for each $k\in\NN$, let $d_k\df\lvert\{w\in V_2(G) \mid
  d_G(w) = k\}\rvert$ be the number of vertices in $V_2(G)$ that have degree $k$.

  If for each $k\geq 2$, we have either $d_k = 0$ or $d_k\geq\binom{v_1(G)}{k}$, then $G$ is a strong
  Sidorenko bigraph.
\end{restatable}

\begin{restatable}{theorem}{leftweakHolderorbits}\label{thm:leftweakHolderorbits}
  Let $H$ be a non-trivial right-uniform color-edge-transitive left-weakly \Holder\ bigraph without isolated
  vertices and let $G$ be a bigraph with $V_1(G) = V_1(H)$ and without isolated vertices. For every
  $U\subseteq V_1(G)$, let
  \begin{align*}
    d_G(U) & \df \lvert\{w\in V_2(G) \mid N_G(w) = U\}\rvert,\\
    d_H(U) & \df \lvert\{w\in V_2(H) \mid N_H(w) = U\}\rvert.
  \end{align*}
  Suppose further that for every $U\subseteq V_1(G)$ with $\lvert U\rvert\geq 2$ the following hold.
  \begin{enumerate}
  \item $\sum_{\sigma\in\Aut(H)} d_G(\sigma(U)) = 0$ if and only if $\sum_{\sigma\in\Aut(H)} d_H(\sigma(U)) =
    0$.
    \label{thm:leftweakHolderorbits:zero}
  \item $\sum_{\sigma\in\Aut(H)} d_G(\sigma(U)) \geq \sum_{\sigma\in\Aut(H)} d_H(\sigma(U))$.
    \label{thm:leftweakHolderorbits:geq1}
  \end{enumerate}

  Then $G$ is a strong Sidorenko bigraph. In particular, $G(H)$ is a strong Sidorenko bigraph.
\end{restatable}

The most useful examples of right-uniform color-edge transitive left-weakly \Holder\ bigraphs to be used in
the theorem above are obtained from left-reflection bigraphs through
Theorem~\ref{thm:leftrefl->leftweakHolder}.

\section{Left-sided properties}
\label{sec:leftsided}

In this section, we show the theorems that do not directly involve Sidorenko's Conjecture, that is,
Theorems~\ref{thm:leftrefl->leftcutperc}, \ref{thm:leftcutperc->leftweakHolder}
and~\ref{thm:incidencecomplete}. These theorems are mostly direct analogues of Conlon--Lee~\cite{CL17}.

Theorem~\ref{thm:leftrefl->leftcutperc} will be easily derived from the following property of
left-cut-percolating sequences.
\begin{lemma}\label{lem:amalgleftcutperc}
  Let $G_1,\ldots,G_k$ be bigraphs with $V_1(G_1) = \cdots = V_1(G_k)$ and $V_2(G_1),\ldots,V_2(G_k)$ pairwise
  disjoint and let $G$ be the amalgamation of $G_1,\ldots,G_k$ over the left side, i.e., we have
  \begin{align*}
    V_1(G) & \df V_1(G_1), &
    V_2(G) & \df \bigcup_{i\in[k]} V_2(G_i), &
    E(G) & \df \bigcup_{i\in[k]} E(G_i).
  \end{align*}

  Suppose further that $G_1$ has a left-cut-percolating sequence $U_0,U_1,\ldots,U_m$ with respect
  to $\Phi = ((\phi_1,L_1),\ldots,(\phi_m,L_m))$ and for every $i\in[m]$ and $j\in\{2,\ldots,k\}$,
  there exists a fold $(\psi_{ij},L_{ij})$ of $G_j$ such that $\phi_i\rest_{V_1(G_1)} =
  \psi_{ij}\rest_{V_1(G_j)}$ and $L_i\cap V_1(G_1) = L_{ij}\cap V_1(G_j)$.

  Then $U_0,U_1,\ldots,U_m$ is a left-cut-percolating sequence in $G$ with respect to
  $\widehat{\Phi} = ((\widehat{\phi}_1,\widehat{L}_1),\ldots,(\widehat{\phi}_m,\widehat{L}_m))$,
  where $\widehat{L}_i\df L_i\cup\bigcup_{j=2}^k L_{ij}$ and 
  \begin{align*}
    \widehat{\phi}_i(v) & =
    \begin{dcases*}
      \phi_i(v), & if $v\in V(G_1)$,\\
      \psi_{ij}(v), & if $v\in V(G_j)$, $j\in\{2,\ldots,k\}$
    \end{dcases*}
  \end{align*}
  is the amalgamation of $\phi_i,\psi_{2k},\ldots,\psi_{ik}$.
\end{lemma}

\begin{proof}
  Since $(\widehat{\phi}_i)_{\widehat{L}_i}\rest_{V_1(G)} = (\phi_i)_{L_i}\rest_{V_1(G_1)}$, it will follow that
  $U_0,U_1,\ldots,U_m$ is a left-cut-percolating sequence of $G$ with respect to $\widehat{\Phi}$ as long as
  we show that $\widehat{\Phi}$ is indeed a sequence of folds of $G$.
  
  Fix $i\in[m]$, write $\psi_{i1}\df\phi_i$ and $L_{i1}\df L_i$. The fact that $\widehat{\phi}_i$ is an
  involution of $G$ follows simply because the functions $\psi_{i1},\ldots,\psi_{ik}$ are involutions of their
  respective bigraphs and they coincide in the common part $V_1(G_1)$.

  From the same property, it also follows that $\Fix(\widehat{\phi}_i) = \bigcup_{j=1}^k\Fix(\psi_{ij})$ and
  $\widehat{\phi}_i(\widehat{L}_i) = \bigcup_{j=1}^k \psi_{ij}(L_{ij})$ and thus $\widehat{\phi}_i$ is a
  cut-involution of $G$ and $(\widehat{L}_i,\Fix(\widehat{\phi}_i),\widehat{\phi}_i(\widehat{L}_i))$ forms a
  partition of $V(G)$.

  It remains to show that $G\rest_{\widehat{L}_i}$ is a union of connected components of $G -
  \Fix(\widehat{\phi}_i)$. To show this, it is sufficient to show that if $v_1,v_2\in
  V(G)\setminus\Fix(\widehat{\phi}_i)$ are in the same component of $G - \Fix(\widehat{\phi}_i)$ and
  $v_1\in\widehat{L}_i$, then $v_2\in\widehat{L}_i$. But indeed, if we partition any path $P$ from
  $v_1$ to $v_2$ into segments $P_1,\ldots,P_n$ such that each segment $P_t$ is completely contained
  in $V(G_{j_t})$ for some $j_t\in[k]$, then $P_t$ must be entirely contained in $L_{j_t}$ as
  $L_{j_t}$ is a union of connected components of $G_{j_t} - \Fix(\psi_{ij})$ (and $\Fix(\psi_{ij})
  = \Fix(\widehat{\phi}_i)\cap V(G_j)$). Therefore $(\widehat{\phi}_i,\widehat{L}_i)$ is a fold of
  $G$.
\end{proof}

We now prove Theorem~\ref{thm:leftrefl->leftcutperc}, which says that every left-reflection bigraph is
left-cut-percolating under reflection folds.

\begin{proofof}{Theorem~\ref{thm:leftrefl->leftcutperc}}
  Let $G$ be an $(S_0;S_1,\ldots,S_k;S_\Delta)$-left-reflection bigraph relative to the reflection group $R$
  with set of reflections $T$ and for each $j\in[k]$ let $G_j$ be the $(S_0,S_j;S_\Delta)$-reflection
  bigraph.

  By~\cite[Corollary~4.9 and Theorem~4.12]{CL17}, there exists a cut-percolating sequence
  $E_0,\ldots,E_m$ with respect to a sequence $\Phi=((\phi_1,L_1),\ldots,(\phi_m,L_m))$ of \emph{reflection
    folds} of $G_1$ ($i\in[m]$).

  Since $G_1$ does not have isolated vertices, by Remark~\ref{rmk:cutperc->leftcutperc}, setting $U_i\df
  E_i\cap V_1(G_1)$ ($i\in\{0,1,\ldots,m\}$) gives a left-cut-percolating sequence of $G_1$ with respect to
  $\Phi$.

  For each $i\in [m]$, let $t_i\in T$ be a reflection defining the reflection fold $(\phi_i,L_i)$ of $G_1$,
  that is, we have $\phi_i = \phi(t_i;S_0,S_1;S_\Delta)$ and $L_i = L(t_i;S_0,S_1;S_\Delta)$ and for each
  $j\in\{2,\ldots,k\}$, let $(\psi_{ij},L_{ij})\df (\phi(t_i;S_0,S_j;S_\Delta), L(t_i;S_0,S_j;S_\Delta))$ be
  the reflection fold of $G_j$ defined by the same reflection $t_i$. Then the hypotheses of
  Lemma~\ref{lem:amalgleftcutperc} are satisfied and we deduce that $G$ is left-cut-percolating with respect
  to a sequence of reflection folds.
\end{proofof}

To prove Theorem~\ref{thm:leftcutperc->leftweakHolder}, we need to recall another definition and lemma
from~\cite{CL17}.

\begin{definition}[slightly adapted from Conlon--Lee~\protect{\cite[\S 3]{CL17}}]
  Let $G$ be bigraph, let $c\function{E(G)}{C}$ be a coloring of $G$ and let
  $\Phi=((\phi_1,L_1),\ldots,(\phi_m,L_m))$ be a sequence of folds of $G$. The \emph{Cauchy--Schwarz tree}
  rooted at $(G,c)$ corresponding to $\Phi$ is the rooted complete binary tree $T(G,c,\Phi)$ of height $m$ in
  which each vertex is labeled by a coloring of $c$ so that:
  \begin{enumerate}
  \item The root of $T(G,c,\Phi)$ is labeled by $c$.
  \item If a node at height $i\in\{0,1,\ldots,m-1\}$ is labeled by $c'$, then its left and right children are
    labeled by $c\comp(\phi_i)_{L_i}$ and $c\comp(\phi_i)_{L_i}^*$, respectively.
  \end{enumerate}
\end{definition}

\begin{lemma}[adapted from Conlon--Lee~\protect{\cite[\S 3]{CL17}}]\label{lem:CStree}
  Let $G$ be a bigraph, let $c\function{E(G)}{C}$ be a coloring of $G$, let $(\phi,L)$ be a fold of $G$ and
  let $W = (W_j)_{j\in C}$ be a sequence of bigraphons. Then
  \begin{align*}
    t((G,c),W) & \leq
    t((G,c\comp\phi_L),W)^{1/2}\cdot
    t((G,c\comp\phi_L^*),W)^{1/2},
  \end{align*}

  More generally, if $\Phi=((\phi_1,L_1),\ldots,(\phi_m,L_m))$ is a sequence of folds of $G$ and
  $c_1,\ldots,c_{2^m}$ are the colorings (with multiplicities) that label the leaves of the Cauchy--Schwarz
  tree $T(G,c,\Phi)$, then
  \begin{align*}
    t((G,c),W) & \leq \prod_{t=1}^{2^m} t((G,c_t),W)^{2^{-m}}.
  \end{align*}
\end{lemma}

\begin{proof}
  For the first statement, let
  \begin{align*}
    G_0 & \df G\rest_{\Fix(\phi)}, &
    G_1 & \df G\rest_{L\cup\Fix(\phi)} - E(G_0), &
    G_2 & \df G\rest_{\phi(L)\cup\Fix(\phi)} - E(G_0),\\
    c_0 & \df c\rest_{E(G_0)}, &
    c_1 & \df c\rest_{E(G_1)}, &
    c_2 & \df c\rest_{E(G_2)},\\
    F_0 & \df ((G_0,c_0),\Fix(\phi)), &
    F_1 & \df ((G_1,c_1),\Fix(\phi)), &
    F_2 & \df ((G_2,c_2),\Fix(\phi)).
  \end{align*}
  By Cauchy--Schwarz Inequality, we have
  \begin{align*}
    t((G,c),W)
    & =
    \int_{X^{V_1}\cap Y^{V_2}}
    t(F_0,W)(x,y)\cdot t(F_1,W)(x,y)\cdot t(F_2,W)(x,y)
    \ d(\mu\otimes\nu)(x,y)
    \\
    & \leq
    \begin{multlined}[t]
      \left(\int_{X^{V_1}\cap Y^{V_2}}
      t(F_0,W)(x,y)\cdot t(F_1,W)(x,y)^2
      \ d(\mu\otimes\nu)(x,y)\right)^{1/2}
      \\
      \cdot
      \left(\int_{X^{V_1}\cap Y^{V_2}}
      t(F_0,W)(x,y)\cdot t(F_2,W)(x,y)^2
      \ d(\mu\otimes\nu)(x,y)\right)^{1/2}
    \end{multlined}
    \\
    & =
    t((G,c\comp\phi_L),W)^{1/2}\cdot
    t((G,c\comp\phi_L^*),W)^{1/2},
  \end{align*}
  where $V_i\df V_i(G)\cap\Fix(\phi)$.

  The statement for Cauchy--Schwarz trees follows by induction.
\end{proof}

We now show Theorem~\ref{thm:leftcutperc->leftweakHolder}, whose statement is repeated below.

\leftcutperctoleftweakHolder*

\begin{proofof}{Theorem~\ref{thm:leftcutperc->leftweakHolder}}
  Let $U_0,\ldots,U_m$ be a left-cut-percolating sequence of $G$ with respect to a sequence $\Phi\df
  ((\phi_1,L_1),\ldots,(\phi_m,L_m))$ of folds in $S$ and for each $t\in\NN_+$, let $\Phi^t$ be the
  concatenation of $\Phi$ with itself $t$ times. Let also $\widehat{S}$ be the subgroup of $\Aut(G)$ generated
  by $\{\phi \mid (\phi,L)\in S\}$.

  Let us call a coloring of $G$ left-constant if it is of the form $\ell\otimes c$ for some \emph{constant}
  left-coloring $\ell\function{V_1(G)}{C'}$.

  We claim that for any left-coloring $\ell\function{V_1(G)}{C'}$, at least a $1 - (1 - 2^{-m})^t$ fraction
  of the leaves of the Cauchy--Schwarz tree $T(G,\ell\otimes c,\Phi^t)$ are labeled by left-constant colorings
  of $G$.

  For $t = 1$, since $U_0,\ldots,U_m$ is a left-cut-percolating sequence, the label of the leftmost leaf of
  $T(G,\ell\otimes c,\Phi)$ is
  \begin{align*}
    c_0
    & \df
    (\ell\otimes c)\comp(\phi_1)_{L_1}\comp\cdots\comp(\phi_m)_{L_m}
    =
    (\ell\comp(\phi_1)_{L_1}\comp\cdots\comp(\phi_m)_{L_m})\otimes c,
  \end{align*}
  where the equality follows from the fact that $c$ is $\widehat{S}$-invariant. This means that if $v_0$ is
  the unique element of $U_0$, then for every $(v,w)\in E(G)$, we have
  \begin{align*}
    c_0(v,w)
    & =
    ((\ell\comp(\phi_1)_{L_1}\comp\cdots(\phi_m)_{L_m})(v),c(v,w))
    =
    (\ell(v_0),c(v,w)),
  \end{align*}
  so $c_0 = \ell(v_0)\otimes c(v,w)$. Therefore at least a $2^{-m} = 1 - (1 - 2^{-m})^1$ fraction of leaves are
  labeled by constant left-colorings

  Suppose now that $t\geq 2$ and note that if a node of $T(G,\ell\otimes c,\Phi^t)$ is labeled by a
  left-constant coloring $c_0$, then all of its descendants must also be labeled by $c_0$. By induction, we
  also know that at least a $1 - (1 - 2^{-m})^{t-1}$ fraction of the nodes at level $(t-1)m$ are labeled by
  left-constant colorings. On the other hand, the case $t=1$ above also guarantees that for each node at level
  $(t-1)m$, at least one of its descendant leaves is labeled by a left-constant coloring. This means that the
  fraction of leaves of $T(G,\ell\otimes c,\Phi^t)$ labeled by left-constant colorings is at least
  \begin{align*}
    (1 - (1 - 2^{-m})^{t-1}) + \frac{(1 - 2^{-m})^{t-1}}{2^m}
    & =
    1 - (1 - 2^{-m})^t,
  \end{align*}
  as desired.

  \medskip

  Note now that the definition of $T(G,\ell\otimes c,\Phi^t)$ ensures that each of the left-constant
  colorings that appears in the leaves must be of the form $\ell(v)\otimes c$ for some $v\in V_1(G)$. This
  means that if $\cC$ is the set of colorings $c'\function{E(G)}{C'\times C}$ that are \emph{not}
  left-constant, then applying Lemma~\ref{lem:CStree} to $T(G,\ell\otimes c,\Phi^t)$ and using the claim
  above, we get that
  \begin{align*}
    t((G,\ell\otimes c),W)
    & \leq
    \prod_{v\in V_1(G)} t((G,\ell(v)\otimes c),W)^{\alpha^t_v} \cdot\prod_{c'\in\cC} t((G,c'),W)^{\beta^t_{c'}}
  \end{align*}
  for some $\alpha^t_v\in [0,1]$ and $\beta^t_c\in[0,1]$ satisfying
  \begin{align*}
    \sum_{v\in V_1(G)} \alpha^t_v + \sum_{c'\in\cC} \beta^t_{c'} & = 1, &
    \sum_{v\in V_1(G)} \alpha^t_v & \geq 1 - (1 - 2^{-m})^t.
  \end{align*}

  Letting $t\to\infty$ along some subsequence such that $(\alpha^t_v)_t$ is convergent for every $v\in
  V_1(G)$, we get
  \begin{align*}
    t((G,\ell\otimes c),W)
    & \leq
    \prod_{v\in V_1(G)} t((G,\ell(v)\otimes c),W)^{\alpha_v}
  \end{align*}
  for some $\alpha_v\geq 0$ such that $\sum_{v\in V_1(G)}\alpha_v = 1$, since we have $\beta_t^{c'}\leq
  (1-2^{-m})^t\xrightarrow{t\to\infty} 0$ for every $c'\in\cC$.

  Recall from Remark~\ref{rmk:transitivity} that $G$ is left-vertex-transitive under the action of
  $\widehat{S}$. Then, by Remark~\ref{rmk:Autcolor} we get
  \begin{align*}
    t((G,\ell\otimes c),W)
    & =
    \prod_{\sigma\in\widehat{S}} t((G,(\ell\otimes c)\comp\sigma),W)^{1/\lvert\widehat{S}\rvert}
    \\
    & \leq
    \prod_{\sigma\in\widehat{S}} \prod_{v\in V_1(G)} t((G,(\ell(v)\otimes c)\comp\sigma),W)^{\alpha_v/\lvert\widehat{S}\rvert}
    \\
    & =
    \prod_{\sigma\in\widehat{S}} \prod_{v\in V_1(G)} t((G,(\ell(\sigma(v))\otimes c)),W)^{\alpha_v/\lvert\widehat{S}\rvert}
    \\
    & =
    \prod_{v\in V_1(G)} t((G,\ell(v)\otimes c),W)^{1/v_1(G)},
  \end{align*}
  where the second equality follows since $c$ is $\widehat{S}$-invariant and the third equality follows since
  $G$ is left-vertex-transitive under the action of $\widehat{S}$. Therefore $(G,c)$ is left-weakly \Holder.
\end{proofof}

Finally, for Theorem~\ref{thm:incidencecomplete}, which says that the incidence bigraph of the complete
hypergraph on $n$ vertices and in uniformities $k_1,\ldots,k_t$ is a left-reflection bigraph, we will heavily
rely on the fact that the incidence bigraph of complete $k_i$-uniform hypergraph is a reflection bigraph.

\begin{proofof}{Theorem~\ref{thm:incidencecomplete}}
  Recall from~\cite[Example~4.4]{CL17} that the symmetric group $\fS_n$ on $n$ points with its natural
  embedding in $\GL_n(\RR)$ is generated by the transpositions, which are the reflections of $\fS_n$. By using
  the canonical ordered basis $U\df(e_1,\ldots,e_n)$, the set of positive roots of $\fS_n$ is precisely
  \begin{align*}
    \Phi^+ & = \left\{\frac{e_i - e_j}{\sqrt{2}} \;\middle\vert\; 1\leq i < j\leq n\right\},
  \end{align*}
  the set of simple roots is
  \begin{align*}
    \Delta & = \left\{\frac{e_i - e_{i+1}}{\sqrt{2}} \;\middle\vert\; i\in [n-1]\right\},
  \end{align*}
  and the set of simple reflections is
  \begin{align*}
    S_\Delta & = \{t_{i,i+1} \mid i\in[n-1]\},
  \end{align*}
  where $t_{i,j}$ is the transposition that swaps $i$ and $j$.

  For $k\in[n]$, let $S_k\df S_\Delta\setminus\{t_{k,k+1}\}$, which generates the naturally embedded subgroup
  $R_k\df\fS_k\times\fS_{n-k}$. The left-cosets of $R_k$ can be identified with $\binom{[n]}{k}$ naturally via
  $\sigma R_k\mapsto \sigma([k])$, which means that the $(S_1,S_k;S_\Delta)$-reflection bigraph is isomorphic
  to the incidence bigraph of the complete $k$-uniform hypergraph $H$ on $n$ vertices given by
  \begin{align*}
    V_1(H) & \df [n], &
    V_2(H) & \df \binom{[n]}{k}, &
    E(H) & \df \left\{(i,A)\in [n]\times\binom{[n]}{k} \;\middle\vert\; i\in A\right\},
  \end{align*}
  and the isomorphism is given by
  \begin{align*}
    \fS_n/R_1\ni \sigma R_1 & \mapsto \sigma(1)\in V_1(H),\\
    \fS_n/R_k\ni \sigma R_k & \mapsto \sigma([k])\in V_1(H).
  \end{align*}

  Thus the $(S_1;S_{k_1},\ldots,S_{k_t};S_\Delta)$-left-reflection bigraph is isomorphic to the incidence
  bigraph $G$ of the complete hypergraph on $n$ vertices and in uniformities $k_1,\ldots,k_t$, with the
  isomorphism given by
  \begin{align*}
    \fS_n/R_1\ni \sigma R_1 & \mapsto \sigma(1)\in V_1(G),\\
    \fS_n/R_{k_i}\ni \sigma R_{k_i} & \mapsto (\sigma([k]), j)\in V_2(G).
    \qedhere
  \end{align*}
\end{proofof}

We conclude this section proving Theorem~\ref{thm:leftrefl->leftweakHolder}, which says that left-reflection
bigraphs become right-uniform color-edge-transitive left-weakly \Holder\ bigraphs when equipped with their
natural coloring.

\begin{proofof}{Theorem~\ref{thm:leftrefl->leftweakHolder}}
  The fact that $G$ is left-weakly \Holder\ follows from Theorems~\ref{thm:leftrefl->leftcutperc}
  and~\ref{thm:leftcutperc->leftweakHolder} by noting that the natural coloring of a left-reflection bigraph
  is invariant under the subgroup generated by the cut-involutions coming from reflection folds.

  It is also obvious from the definition of the natural coloring that it is right-uniform. Finally, since each
  of the color classes of $G$ yields a reflection bigraph, which is cut-percolating under reflection folds
  by~\cite[Corollary~4.9 and Theorem~4.12]{CL17}, hence edge-transitive with respect to cut-involutions coming
  from reflection folds (see Remark~\ref{rmk:transitivity}), it follows that $G$ is color-edge-transitive
  under the action of the group generated by cut-involutions coming from reflection folds.
\end{proofof}

\section{Induced-Sidorenko}
\label{sec:indSid}

In this section, we prove Theorem~\ref{thm:leftcutperc->indSid} that says that every left-cut-percolating
bigraph is induced-Sidorenko.

We start with a lemma that is often viewed as an obstacle for inequalities concerning densities in
bigraphons. However, this lemma can also be used positively to deduce equalities of exponents (see for
example its use in Remark~\ref{rmk:leftweakHolder->Sid}).

\begin{lemma}\label{lem:obstacle}
  Let $W$ be a non-zero bigraphon, let $G_1,\ldots,G_n$ be bigraphs and let $r_1,\ldots,r_n\in\RR$. If for
  every $\lambda > 0$, we have
  \begin{align*}
    \prod_{i=1}^n t(G_i,\lambda W)^{r_i} & \geq 1,
  \end{align*}
  then $\sum_{i=1}^n r_i\cdot e(G_i) = 0$.
\end{lemma}

\begin{proof}
  Since $t(G_i,\lambda W) = \lambda^{e(G_i)}\cdot t(G_i,W)$, we must have
  \begin{align*}
    \lambda^{\sum_{i=1}^n r_i\cdot e(G_i)}\prod_{i=1}^n t(G_i,W)^{r_i} \geq 1
  \end{align*}
  for every $\lambda > 0$. Since $W$ is non-zero, we have $t(G_i,W) > 0$. If the exponent of
  $\lambda$ is positive, then making $\lambda$ small enough violates the inequality above. If the
  exponent of $\lambda$ is negative, then making $\lambda$ large enough violates the inequality
  above.
\end{proof}

We now introduce the notion of $2$-threshold subgraphs.

\begin{definition}
  Let $G$ be a bigraph and let $f\function{V(G)}{\{0,1,2\}}$, the \emph{$2$-threshold subgraph of $G$
    corresponding to $f$} is the spanning subgraph $G_f$ of $G$ defined by
  \begin{gather*}
    V_1(G_f) \df V_1(G), \qquad
    V_2(G_f) \df V_2(G),
    \\
    E(G_f) \df \{(v,w)\in E(G) \mid f(v) + f(w)\geq 2\}.
  \end{gather*}
  We let $T_2(G)$ be the set of all $2$-threshold subgraphs of $G$.

  Given an endomorphism $\phi\in\End(G)$ of $G$ and a spanning subgraph $G'$ of $G$, we let $\phi^{-1}(G')$ be
  the spanning subgraph of $G$ defined by
  \begin{gather*}
    V_1(\phi^{-1}(G')) \df V_1(G), \qquad
    V_2(\phi^{-1}(G')) \df V_2(G),
    \\
    E(\phi^{-1}(G')) \df \phi^{-1}(E(G')) = \{(v,w)\in E(G) \mid (\phi(v),\phi(w))\in E(G')\}.
  \end{gather*}
\end{definition}

\begin{remark}\label{rmk:2thresh}
  It is easy to see that if $f = \One_V$ for some $V\subseteq V(G)$, then $G_f = G - (E(G)\setminus
  E(G\rest_V))$, that is, $2$-threshold subgraphs corresponding to $\{0,1\}$-valued functions
  essentially capture induced subgraphs (except for the presence of extra isolated vertices to make
  them spanning).

  It is also easy to see that for an endomorphism $\phi\in\End(G)$ of $G$ and $f\function{V(G)}{\{0,1,2\}}$,
  we have $\phi^{-1}(G_f) = G_{f\comp\phi}$.
\end{remark}

\begin{lemma}\label{lem:2thresh}
  Let $G$ be a left-cut-percolating bigraph and let
  \begin{align*}
    T & \df \{G_f\in T_2(G) \mid f\function{V(G)}{\{0,1,2\}}, f^{-1}(2)\subseteq V_1(G)\}.
  \end{align*}

  Then there exists $m\in\NN_+$ such that for every $H\in T$, there exist $\ell_H\in\NN$ with $\ell_H\leq
  e(G)/v_1(G)$ and a function $r_H\function{T}{\RR_+}$ such that $\sum_{H'\in T} r_H(H') = 1 - 2^{-m}$ and
  \begin{align*}
    t(H,W) & \leq \frac{t(G,W)^{1/2^m}}{t(\rho,W)^{\ell_H}}\cdot\prod_{H'\in T} t(H',W)^{r_H(H')}.
  \end{align*}
  for every biregular non-zero bigraphon $W$.
\end{lemma}

\begin{proof}
  Let $U_0,\ldots,U_m$ be a left-cut-percolating sequence of $G$ with respect to a sequence of folds
  $\Phi=((\phi_1,L_1),\ldots,(\phi_m,L_m))$ and let $v_0$ be the unique element of $U_0$.

  Fix an element $H = G_f$ ($f\function{V(G)}{\{0,1,2\}}$ with $f^{-1}(2)\subseteq V_1(G)$) of $T$ and let us
  first prove the case when $f(v_0) = 2$ and in this case, we will show that we can take $\ell_H =
  0$.

  Note that there is a natural one-to-one correspondence between colorings $c\function{E(G)}{\{0,1\}}$ of $G$
  and spanning subgraphs of $G$ in which a coloring $c$ corresponds to the spanning subgraph $G^c$ given by
  $E(G^c)\df c^{-1}(1)$ and a spanning subgraph $G'$ of $G$ corresponds to the coloring
  $c_{G'}\df\One_{E(G')}$. Consider the sequence $W' = (W'_0,W'_1)$ of bigraphons, where $W_0'\df 1$ and
  $W_1'\df W$ and note that $t((G,c),W') = t(G^c,W)$.

  Consider now the Cauchy--Schwarz tree $T(G,c_H,\Phi)$ and since the folding maps are endomorphisms of $G$,
  by Remark~\ref{rmk:2thresh}, all of the nodes of this tree are labeled by colorings of the form $c_{H'}$ for
  some $H'\in T$. Note further that by the same remark, the leftmost leaf of $\Phi$ has label $c_0\df
  c_{G_g}$, where $g\function{V(G)}{\{0,1\}}$ is given by
  \begin{align*}
    g & \df f\comp(\phi_1)_{L_1}\comp\cdots(\phi_m)_{L_m},
  \end{align*}
  and since $U_0,\ldots,U_m$ is left-cut-percolating, we have $g(v) = g(v_0) = 2$ for every $v\in V_1(G)$,
  which implies $G^{c_0} = G_g = G$. By Lemma~\ref{lem:CStree}, it follows that
  \begin{align*}
    t(H,W)
    & =
    t((G,c_H),W)
    \leq
    t(G,W)^{1/2^m}\cdot\prod_{H'\in T} t(H',W)^{r_H(H')},
  \end{align*}
  where $r_H(H')$ is the number of non-leftmost leaves of $T(G,c_H,\Phi)$ that are labeled by $c_{H'}$ divided
  by $2^m$ (note that $c_G$ can also appear as the label of non-leftmost leaves) and thus $\sum_{H'\in T}
  r_H(H') = 1 - 2^{-m}$.

  \medskip

  We now consider the case when $f(v) = 2$ for some $v\in V_1(G)$ (but not necessarily $v=v_0$) and we will
  also show that in this case we can take $\ell_H = 0$. By Remark~\ref{rmk:transitivity}, we know that $G$ is
  left-vertex-transitive, so letting $\psi\in\Aut(G)$ be an automorphism with $\psi(v_0) = v$, from 
  Remarks~\ref{rmk:Autcolor} and~\ref{rmk:2thresh}, we get
  \begin{align*}
    t(H,W)
    & =
    t((G,c_H),W')
    =
    t((G,c_H\comp\psi),W')
    \\
    & =
    t(\psi^{-1}(G_f),W)
    =
    t(G_{f\comp\psi},W)
  \end{align*}
  so the result follows from the previous case as $(f\comp\psi)(v_0) = 2$.

  \medskip

  Let us now show the case when $\im(f)\subseteq\{0,1\}$ and $f(v_0) = 1$. Let $f'\function{V(G)}{\{0,1,2\}}$
  be the function obtained from $f$ by changing the value of $v_0$ to $2$ and let
  \begin{align*}
    \ell_H & \df \lvert N_G(v_0)\cap f^{-1}(0)\rvert \leq d_G(v_0) = \frac{e(G)}{v_1(G)}
  \end{align*}
  be the number of neighbors of $v_0$ that have value $0$ under $f$ (the last equality in the above
  follows from left-vertex-transitivity). Note that $G_{f'}$ is obtained from $G_f$ by precisely
  adding the $\ell_H$ edges from $v_0$ to its neighbors in $G$ whose value under $f$ is $0$, and
  since $\im(f)\subseteq\{0,1\}$, the right endpoint of these newly added edges were isolated
  vertices in $G_f$, hence in $G_{f'}$ they have only $v_0$ as their neighbor. Since $W$ is
  biregular, it follows that
  \begin{align*}
    t(G_{f'},W) & = t(G_f,W)\cdot t(\rho,W)^{\ell_H}.
  \end{align*}
  Since $f'(v_0) = 2$, the result now follows from the first case (note that the upper bound on $\ell_H$ for
  this case is $e(G)/v_1(G)$ since the first case was shown with $\ell_H = 0$).

  \medskip

  The case when $\im(f)\subseteq\{0,1\}$ and there exists $v\in V_1(G)$ with $f(v)=1$ follows by
  left-vertex-transitivity from the previous case.

  \medskip

  The final case is when $f\rest_{V_1(G)} = 0$. But since $f(V_2(G))\subseteq\{0,1\}$, it follows that $H$ is
  empty and since $G$ is a Sidorenko bigraph (by Theorem~\ref{thm:leftcutperc->leftweakHolder} and
  Remark~\ref{rmk:leftweakHolder->Sid}), we get
  \begin{align*}
    t(H,W)
    & =
    1
    \leq
    \frac{t(G,W)^{1/2^m}}{t(\rho,W)^{e(G)/2^m}}
    =
    \frac{t(G,W)^{1/2^m}}{t(\rho,W)^{e(G)/2^m}}\cdot t(H,W)^{1 - 2^{-m}},
  \end{align*}
  which means that the result will follow by setting $r_H(H')\df \One[H'=H] (1 - 2^{-m})$ and $\ell_H\df
  e(G)/2^m$ as long as we prove the bound $\ell_H\leq e(G)/v_1(G)$, that is, we need to show that $v_1(G)\leq
  2^m$. But note that from the definition of left-cut-percolating sequence, we have $\lvert U_i\rvert\leq
  2\lvert U_{i-1}\rvert$ for every $i\in[m]$, hence a simple induction gives $v_1(G) = \lvert U_m\rvert\leq
  2^m\lvert U_0\rvert = 2^m$, as desired.
\end{proof}

We conclude this section proving Theorem~\ref{thm:leftcutperc->indSid}, which says that every
left-cut-percolating bigraph is induced-Sidorenko.

\begin{proofof}{Theorem~\ref{thm:leftcutperc->indSid}}
  Let $G$ be a left-cut-percolating bigraph and $G'$ be an induced bigraph of $G$ and let us show that $G$
  weakly dominates $G'$.

  Let $T$ and $m$ be as
  in Lemma~\ref{lem:2thresh} and for each $H\in T$, let $\ell_H$ and $r_H$ also be as in the same lemma. Let
  also $p\df 1-2^{-m}$.

  Let $H\df G_{\One_{V(G')}}$ be the $2$-threshold subgraph of $G$ corresponding to the indicator function
  $\One_{V(G')}$ of $V(G')$. Note that $H\in T$ and since $E(H) = E(G')$, for every biregular non-zero
  bigraphon $W$, we have $t(G',W) = t(H,W)$, so to show that $G$ weakly dominates $G'$, we can show that it
  weakly dominates $H$ (see Remark~\ref{rmk:2thresh}).

  We construct a sequence $(\alpha_n)_{n\in\NN}$ of non-negative reals and a sequence
  $(r_n)_{n\in\NN}$ of functions $r_n\function{T}{\RR_+}$ such that the following hold.
  \begin{enumerate}
  \item For every $n\in\NN$ and every biregular non-zero bigraphon $W$, we have
    \begin{align}\label{eq:leftcutperc->indSid:tHW}
      t(H,W) & \leq \frac{t(G,W)^{1 - p^n}}{t(\rho,W)^{\alpha_n}} \prod_{H'\in T} t(H',W)^{r_n(H')}.
    \end{align}
  \item For every $n\in\NN$, we have
    \begin{align*}
      \sum_{H'\in T} r_n(H') = p^n.
    \end{align*}
  \item For every $n\in\NN$, we have
    \begin{align*}
      \alpha_n \leq \alpha_{n+1} \leq \alpha_n + p^n\cdot\frac{e(G)}{v_1(G)}.
    \end{align*}
  \end{enumerate}

  The construction is by induction: we start with $\alpha_0 = 0$ and $r_0(H') = \One[H=H']$ and for
  $n\in\NN_+$, we define
  \begin{align*}
    \alpha_n & \df \alpha_{n-1} + \sum_{H'\in T} r_{n-1}(H')\cdot\ell_{H'}, &
    r_n(H') & \df \sum_{H''\in T} r_{H''}(H')\cdot r_{n-1}(H'').
  \end{align*}
  The three items follow by induction when we apply Lemma~\ref{lem:2thresh} to all $H'\in T$ on the right-hand
  side of~\eqref{eq:leftcutperc->indSid:tHW}.

  Note that the second item ensures that $\lim_{n\to\infty} r_n(H') = 0$ for every $H'\in T$ and the third
  item ensures that the limit $\alpha\df\lim_{n\to\infty} \alpha_n$ exists as the sequence
  $(\alpha_n)_{n\in\NN}$ is non-decreasing and upper bounded by $e(G)/((1-p)v_1(G))$. By letting $n\to\infty$
  in~\eqref{eq:leftcutperc->indSid:tHW}, we conclude that for every biregular non-zero bigraphon $W$, we have
  \begin{align*}
    t(H,W) & \leq \frac{t(G,W)}{t(\rho,W)^\alpha}
  \end{align*}
  and since this holds for every such $W$, by Lemma~\ref{lem:obstacle}, it follows that $\alpha = e(G) -
  e(H)$, so $G$ weakly dominates $H$ as desired.
\end{proofof}

\section{Symmetrizations and fractional powers of colored bigraphs}
\label{sec:orbits}

In this section we prove Theorems~\ref{thm:largeright} and~\ref{thm:leftweakHolderorbits}. We start by showing
how the former (which is restated below) is a particular case of the latter.

\largeright*

\begin{proofof}{Theorem~\ref{thm:largeright}}
  Without loss of generality, suppose that $V_1(G) = [n]$. Let
  \begin{align*}
    K\df \{d_G(w) \mid w\in V_2(G)\}\setminus\{0,1\} = \{k\in\{2,\ldots,n\} \mid d_k\neq 0\}
  \end{align*}
  and enumerate its elements as $k_1,\ldots,k_t$. Let $G'$ be the incidence bigraph of the complete hypergraph
  on $n$ vertices and in uniformities $k_1,\ldots,k_t$. By Theorem~\ref{thm:incidencecomplete}, we know that
  $G'$ is a left-reflection bigraph. It is straightforward to see that its natural coloring
  $c\function{E(G')}{[t]}$ is the unique function such that $d_{G'}(w) = k_{c(v,w)}$ for every $(v,w)\in
  E(G')$. By Theorem~\ref{thm:leftrefl->leftweakHolder}, we know that $H\df(G,c)$ is left-weakly \Holder.
  Note also that $\Aut(H) = \Aut(G')$ (and is isomorphic to the symmetric group $\fS_n$ on $n$
  points). Furthermore, note that for $U\subseteq [n]$, we have
  \begin{align*}
    \sum_{\sigma\in\Aut(H)} d_G(\sigma(U))
    & =
    \frac{\lvert\Aut(H)\rvert\cdot d_{\lvert U\rvert}}{\binom{n}{\lvert U\rvert}},
    \\
    \sum_{\sigma\in\Aut(H)} d_H(\sigma(U))
    & =
    \begin{dcases*}
      \lvert\Aut(H)\rvert, & if $\lvert U\rvert\in K$,\\
      0, & otherwise.
    \end{dcases*}
  \end{align*}
  Since $d_k\geq\binom{n}{k}$ for every $k\in K$, the result now follows from
  Theorem~\ref{thm:leftweakHolderorbits}.
\end{proofof}

To prove Theorem~\ref{thm:leftweakHolderorbits}, we will need to work with fractional colored bigraphs and a
notion that we call color-Sidorenko defined below. The intuition is that a fractional colored bigraph encodes
the left side of a right-uniform colored bigraph as its vertex set $V$ and counts how many vertices on the right
side have neighborhood in color $i\in C$ exactly equal to a set $U\subseteq V$, except that we allow this
count to be fractional.

\begin{definition}
  A \emph{colored fractional bigraph} is a function $h\function{2^V\times C}{\RR_+}$, where $V$ and $C$ are
  sets, called \emph{vertex set} and \emph{color set} of $h$, respectively. We use the shorthand notations
  $V_h\df V$ and $C_h\df C$ and we let
  \begin{align*}
    v(h) & \df \lvert V_h\rvert,\\
    e_i(h) & \df \sum_{U\subseteq V_h} \lvert U\rvert\cdot h(U,i) \qquad (i\in C),\\
    e(h) & \df \sum_{i\in C} e_i(h).
  \end{align*}
  For each $i\in C_h$ and each $v\in V_h$, the \emph{$i$-degree of $v$ in $h$} is defined as
  \begin{align*}
    d_{h,i}(v) & \df \sum_{\substack{U\subseteq V_h\\ v\in U}} h(U,i).
  \end{align*}
  We say that $h$ is \emph{color-regular} if for every $i\in C_h$ and every $v_1,v_2\in V_h$, we have
  $d_{h,i}(v_1) = d_{h,i}(v_2)$ (which is equivalent to saying that $d_{h,i}(v) = e_i(h)/v(h)$ for every $v\in
  V_h$ and every $i\in C_h$).

  Given a right-uniform colored bigraph $H$ without isolated vertices, its corresponding colored fractional
  bigraph $h_H\function{2^{V_1(H)}\times C_H}{\RR_+}$ is defined by
  \begin{align*}
    h_H(U,i) & \df \lvert\{w\in V_2(H) \mid N_H(w) = U\land \forall v\in N_H(w), c_H(v,w) = i\}\rvert.
  \end{align*}

  Given a colored fractional bigraph $h$ and a sequence $W = (W_i)_{i\in C_h}$ of bigraphons over the same
  spaces $\Omega=(X,\mu)$ and $\Lambda=(Y,\nu)$, we let
  \begin{align*}
    t(h,W)
    & \df
    \int_{X^{V_h}}
    \prod_{U\subseteq V_h} \prod_{i\in C_h}
    t(K_{\lvert U\rvert,1}^L, W_i)(x_U)^{h(U,i)}
    \ d\mu(x).
  \end{align*}
  (Note that this ensures that $t(H,W) = t(h_H,W)$ for right-uniform colored bigraphs $H$.)

  Given a colored fractional bigraph $h$ and a tuple $\vec{p}\df (p_i)_{i\in C_h}\in \RR_+^{C_h}$, the \emph{$\vec{p}$
    color-power} of $h$ is the colored fractional bigraph $h^{\vec{p}}\function{2^{V_h}\times C_h}{\RR_+}$ given by
  \begin{align*}
    h^{\vec{p}}(U,i) & \df h(U,i)\cdot p_i
  \end{align*}
  We extend this definition to right-uniform colored bigraphs $H$ without isolated vertices as $H^{\vec{p}}\df
  h_H^{\vec{p}}$.

  Given a set of colors $C$, the \emph{$C$-rainbow star} is the connected colored bigraph $\rho_C$ with one
  vertex on the left side and one edge of each color in $C$. Formally, it is given by
  \begin{align*}
    V_1(\rho_C) & \df \{1\}, &
    V_2(\rho_C) & \df C, &
    E(\rho_C) & \df \{1\}\times C,
  \end{align*}
  and $c_{\rho_C}\function{E(\rho_C)}{C}$ is defined by
  \begin{align*}
    c_{\rho_C}(1,i) & \df i \qquad (i\in C).
  \end{align*}
  Given a colored fractional bigraph $h$ with $e(h) > 0$, we let $\rho_h\df \rho_{C_h}^{\vec{p}}$, where
  $\vec{p}=(p_i)_{i\in C_h}$ is given by $p_i\df e_i(h)/e(h)$. We extend this definition to right-uniform
  colored bigraphs $H$ by letting $\rho_H\df\rho_{h_H}$.

  A colored fractional bigraph $h$ with $e(h) > 0$ is called \emph{color-Sidorenko} if for every sequence $W =
  (W_i)_{i\in C_h}$ of bigraphons over the same spaces we have
  \begin{align*}
    t(h,W) & \geq t(\rho_h,W)^{e(h)}.
  \end{align*}
  
  A right-uniform colored bigraph $H$ is called \emph{color-Sidorenko} if its corresponding colored fractional
  bigraph $h_H$ is color-Sidorenko.
\end{definition}

We start by showing that right-uniform left-weakly \Holder\ bigraphs $H$ are left-color regular and
color-Sidorenko.

\begin{lemma}\label{lem:leftweakHolder->colorSid}
  If $H=(G,c)$ is a non-trivial left-weakly \Holder\ bigraph without isolated vertices, then $H$ is
  left-color-regular. Furthermore, if $H$ is also right-uniform, then $H$ is color-Sidorenko.
\end{lemma}

\begin{proof}
  Fix $v_0\in V_1(H)$ and a sequence $W = (W_i)_{i\in C_H}$ of bigraphons on the same spaces $\Omega=(X,\mu)$
  and $\Lambda=(Y,\nu)$. Let $\ell\function{V_1(H)}{\{0,1\}}$ be given by $\ell(v)\df \One[v = v_0]$ and
  define the sequence $W' = (W'_{t,i})_{t\in\{0,1\},i\in C_H}$ by
  \begin{align*}
    W'_{t,i} & \df
    \begin{dcases*}
      W_i, & if $t = 1$,\\
      1, & if $t = 0$.
    \end{dcases*}
  \end{align*}
  Then we have
  \begin{align*}
    t((G,0\otimes c),W') & = 1, &
    t((G,1\otimes c),W') & = t(H,W).
  \end{align*}

  Let further $H'\df H\rest_{v_0\cup N_H(v_0)}$ be the restriction of $H$ to $v_0$ and its neighbors and note
  that $t((G,\ell\otimes c),W') = t(H',W)$, so the left-weak \Holder\ property gives $t(H',W) \leq
  t(H,W)^{1/v_1(H)}$, which can be rewritten as
  \begin{align*}
    t(H,W) & \geq t(H',W)^{v_1(H)}.
  \end{align*}

  If we instantiate the above to the case where $W_{i_0} = \widehat{W}$ for some fixed $i_0\in C_H$ and $W_i = 1$
  every $i\neq i_0$, we get
  \begin{align*}
    t(H_{\{i_0\}},\widehat{W}) & \geq t(K_{1,d_{H,i_0}(v_0)},\widehat{W})^{v_1(H)},
  \end{align*}
  which by Lemma~\ref{lem:obstacle} implies that $e_{i_0}(H) = d_{H,i_0}(v_0)\cdot v_1(H)$, that is, $H$ is
  left-color-regular.

  Going back to the general sequence $W$, if $H$ is also right-uniform, then we conclude that
  \begin{align*}
    t(H,W)
    & \geq
    t(H',W)^{v_1(H)}
    \\
    & =
    \left(\int_X \prod_{i\in C_H} \left(\int_Y W_i(x,y)\ d\nu(y)\right)^{d_{H,i}(v_0)}\ d\mu(x)\right)^{v_1(H)}
    \\
    & \geq
    t(\rho_H,W)^{e(H)},
  \end{align*}
  where the last inequality follows from Jensen's Inequality for the convex function $z\mapsto
  z^{e(H)/v_1(H)}$ and the fact that $H$ is left-color-regular. Therefore $H$ is color-Sidorenko.
\end{proof}

Our next objective is to prove that color-powers $H^{\vec{p}}$ of left-weakly \Holder\ bigraphs are also
color-Sidorenko under the further assumptions that $H$ is right-uniform and $p_i\geq 1$ for every $i\in
C_H$. To do so we need to establish several lemmas. We start with one that says that it is sufficient to check
the color-Sidorenko property only for sequences of bigraphons in which all but one are left-regular. The trick
employed here is similar to the one in~\cite{CR21}, except that since $h$ is color-regular, we can perform a
much simpler construction.

\begin{lemma}\label{lem:colorregularcolorSid}
  Let $h$ be a color-regular colored fractional bigraph and let $i_0\in C_h$ be such that $e_{i_0}(h)\neq
  0$. Suppose that $t(h,W)\geq t(\rho_h,W)^{e(h)}$ for every sequence $W = (W_i)_{i\in C_h}$ of positive
  bigraphons over the same spaces such that $W_i$ is left-regular and non-zero for every $i\in
  C_h\setminus\{i_0\}$. Then $h$ is color-Sidorenko.
\end{lemma}

\begin{proof}
  Without loss of generality, we may suppose that $e_i(h)\neq 0$ for every $i\in C_h$.

  By possibly replacing each $W_i$ with $W_i^\epsilon\df \epsilon + W_i$ and applying the Dominated
  Convergence Theorem letting $\epsilon\to 0$, it is sufficient to show that $t(h,W)\geq t(\rho_h,W)^{e(h)}$
  holds for every sequence of bigraphons over the same spaces that are bounded away from $0$. Fix one such
  sequence $W$ of bigraphons over spaces $\Omega=(X,\mu)$ and $\Lambda=(Y,\nu)$ and define the sequence $W' =
  (W'_i)_{i\in C_h}$ by
  \begin{align*}
    W'_i(x,y) & \df
    \begin{dcases*}
      \frac{W_i(x,y)}{t(e_1,W_i)(x)}, & if $i\neq i_0$,\\
      W_{i_0}(x,y)\cdot\prod_{j\in C_h\setminus\{i_0\}} t(e_1,W_j)(x)^{e_j(h)/e_{i_0}(h)}, & if $i = i_0$.
    \end{dcases*}
  \end{align*}
  Note that the fact that each $W_i$ is bounded away from $0$ ensures that the functions above are bounded. It
  is also trivial that $W'_i$ is left-regular for each $i\neq i_0$.

  Note now that
  \begin{align*}
    t(\rho_h,W')
    & =
    \begin{multlined}[t]
      \int_X
      \left(t(e_1,W_{i_0})(x)\cdot
      \prod_{j\in C_h\setminus\{i_0\}} t(e_1,W_j)(x)^{e_j(h)/e_{i_0}(h)}\right)^{e_{i_0}(h)/e(h)}
      \\
      \cdot
      \prod_{i\in C_h\setminus\{i_0\}} \left(\frac{t(e_1,W_i)}{t(e_1,W_i)}\right)^{e_i(h)/e(h)}
      \ d\mu(x)
    \end{multlined}
    \\
    & =
    t(\rho_h,W).
  \end{align*}

  On the other hand, since $h$ is color-regular, we have $d_{h,i}(v) = e_i(h)/v(h)$ for every $i\in C_h$ and
  every $v\in V_h$, so we get
  \begin{align*}
    t(h,W')
    & =
    \begin{multlined}[t]
      \int_{X^{V_h}}
      \prod_{U\subseteq V_h} \prod_{i\in C_h} t(K_{\lvert U\rvert,1}^L, W_i)(x_U)^{h(U,i)}
      \\
      \cdot
      \prod_{v\in V_h}
      \prod_{i\in C_h\setminus\{i_0\}} t(e_1,W_i)(x)^{e_i(h)\cdot d_{h,i_0}(v) /e_{i_0}(h) - d_{h,i}(v)}
    \end{multlined}
    \\
    & = t(h,W)
  \end{align*}
  Therefore, we conclude that
  \begin{align*}
    t(h,W) & = t(h,W') \geq t(\rho_h,W')^{e(h)} = t(\rho_h,W)^{e(h)},
  \end{align*}
  so $h$ is color-Sidorenko.
\end{proof}

The next lemma can be seen as a color version of the induced-Sidorenko property for right-uniform left-weakly
\Holder\ bigraphs.

\begin{lemma}\label{lem:colorrestriction}
  Let $H = (G,c)$ be a right-uniform left-weakly \Holder\ bigraph, let $C\subseteq C_H$ be a subset of
  colors. Let also $W = (W_i)_{i\in C_H}$ be a sequence of bigraphons on the same space such that for each
  $i\in C_H\setminus C$, the bigraphon $W_i$ is left-regular and non-zero. Then
  \begin{align*}
    t(H_C,W) & \leq \frac{t(H,W)}{\prod_{i\in C_H\setminus C} t(\rho,W_i)^{e_i(H)}}.
  \end{align*}
\end{lemma}

\begin{proof}
  Without loss of generality, let us suppose that $H$ does not have isolated vertices and is not trivial.
  
  Fix a vertex $v_0\in V_1(H)$, let $\ell\function{V_1(G)}{\{0,1\}}$ be the left-coloring of $G$ defined by
  $\ell(v)\df \One[v = v_0]$ and let $W' = (W'_{t,i})_{t\in\{0,1\},i\in C_H}$ be given by
  \begin{align*}
    W'_{t,i} & \df
    \begin{dcases*}
      1, & if $t = 0$ and $i\notin C$,\\
      W_i, & if $t = 1$ or $i\in C$.
    \end{dcases*}
  \end{align*}
  Then we have
  \begin{align*}
    t((G,0\otimes c),W') & = t(H_C,W), &
    t((G,1\otimes c),W') & = t(H,W).
  \end{align*}

  Let now $H'=(G',c')$ be the colored bigraph obtained from $G$ by removing all edges with colors in
  $C_H\setminus C$ that are not adjacent to $v_0$, that is, we have $G'\df G - \{(v,w)\in c^{-1}(C_H\setminus
  C) \mid v\neq v_0\}$ and $c'\df c\rest_{E(G')}$. Note that $t((G,\ell\otimes c),W') = t(H',W)$. Note also
  that since $H$ is right-uniform, all edges $(v,w)$ in $E(H')\setminus E(H_C)$ satisfy $v = v_0$, $d_{G'}(w)
  = 1$ and $c(v,w)\in C_H\setminus C$, so since each $W_i$ with $i\in C_H\setminus C$ is left-regular, we get
  \begin{align*}
    t(H',W)
    & =
    t(H_C,W)\cdot \prod_{i\in C_H\setminus C} t(\rho,W_i)^{d_{H,i}(v)}.
  \end{align*}

  Recalling from Lemma~\ref{lem:leftweakHolder->colorSid} that $d_{H,i}(v) = e_i(H)/v_1(H)$, the left-weak
  \Holder\ property gives
  \begin{align*}
    t(H_C,W)\cdot\prod_{i\in C_H\setminus C} t(\rho,W_i)^{e_i(H)/v_1(H)}
    & =
    t(H',W)
    =
    t((G,\ell\otimes c),W')
    \\
    & \leq
    t((G,0\otimes c),W')^{(v_1(H) - 1)/v_1(H)}
    \cdot
    t((G,1\otimes c),W')^{1/v_1(H)}
    \\
    & =
    t(H_C,W)^{(v_1(H) - 1)/v_1(H)}
    \cdot
    t(H,W)^{1/v_1(H)},
  \end{align*}
  so the result follows by taking the $v_1(H)$th power.
\end{proof}

We now prove an inductive form of Jensen's inequality for moments (this can also be seen as inductive form of
\Holder's Inequality).

\begin{proposition}\label{prop:inductiveJensen}
  Let $p_1\geq p_2\geq\cdots\geq p_n\geq 1$ and let $f_1,\ldots,f_n,g\function{\Omega}{\RR_+}$ be bounded
  positive measurable functions in a probability space $\Omega=(X,\mu)$. Then
  \begin{align*}
    \int_X g(x)\cdot\prod_{i=1}^n f_i(x)^{p_i}\ d\mu(x)
    & \geq
    \frac{%
      \left(\int_X g(x)\cdot\prod_{i=1}^n f_i(x)\ d\mu(x)\right)^{p_1}
    }{%
      \prod_{i=1}^n
      \left(\int_X g(x)\cdot\prod_{j=i+1}^n f_j(x)\ d\mu(x)\right)^{p_i-p_{i+1}}
    },
  \end{align*}
  where $p_{n+1}\df 1$ and products of the form $\prod_{j=t}^{t-1}$ are interpreted as $1$.
\end{proposition}

\begin{proof}
  The proof is by induction in $n$. For $n=0$, the result is trivial. The case $n=1$ follows directly from
  Jensen's Inequality for the convex function $z\mapsto z^{p_1}$.

  Suppose then that $n\geq 2$ and that the result holds for $n-1$. Then by the $n = 1$ case with the same $g$
  but taking the whole product as a single function and using the exponent $p_n$, we have
  \begin{align*}
    \int_X g(x)\cdot\prod_{i=1}^n f_i(x)^{p_i}\ d\mu(x)
    & \geq
    \frac{%
      \left(\int_X g(x)\cdot\prod_{i=1}^n f_i(x)^{p_i/p_n}\ d\mu(x)\right)^{p_n}
    }{%
      \left(\int_X g(x)\ d\mu(x)\right)^{p_n - p_{n+1}}
    }.
  \end{align*}

  The result now follows by applying the inductive hypothesis to the integral in the numerator above using
  $g\cdot f_n$ in place of $g$ and exponents $p_1/p_n\geq\cdots\geq p_{n-1}/p_n$ for the functions
  $f_1,\ldots,f_{n-1}$, respectively.
\end{proof}

We can now prove that color-powers of right-uniform left-weakly \Holder\ bigraphs are color-Sidorenko.

\begin{lemma}\label{lem:leftweakHoldercolorpower}
  Let $H$ be a non-trivial right-uniform left-weakly \Holder\ bigraph without isolated vertices and let
  $\vec{p} = (p_i)_{i\in C_H}\in\RR_+^{C_H}$ be such that $p_i\geq 1$ for every $i\in C_H$. Then $H^{\vec{p}}$
  is color-Sidorenko.
\end{lemma}

\begin{proof}
  By possibly removing unused colors and renaming them, we may assume without loss of generality that
  $\im(c_H) = C_H = [n]$ for some $n\in\NN_+$ and that $p_1\geq\cdots\geq p_n\geq 1$.

  Let $h\df H^{\vec{p}}$. By Lemma~\ref{lem:colorregularcolorSid}, it is sufficient to show that $t(h,W)\geq
  t(\rho_h,W)^{e(h)}$ only for sequences $W = (W_i)_{i=1}^n$ of positive bigraphons on the same spaces such
  that all $W_i$ are left-regular except possibly for $W_n$.

  For every $i\in[n]$, let $C_i\df\{i+1,\ldots,n\}$. We now apply Proposition~\ref{prop:inductiveJensen} to
  get
  \begin{align*}
    t(h,W)
    & =
    \int_{X^{V_1(H)}}
    \prod_{i=1}^n t(H_{\{i\}}^L,W)(x)^{p_i}
    \ d\mu(x)
    \\
    & \geq
    \frac{%
      \left(\int_{X^{V_1(H)}} \prod_{i=1}^n t(H_{\{i\}}^L,W)(x) \ d\mu(x)\right)^{p_1}
    }{%
      \prod_{i=1}^n
      \left(\int_{X^{V_1(H)}} \prod_{j=i+1}^n t(H_{\{j\}}^L,W)(x) \ d\mu(x)\right)^{p_i - p_{i+1}}
    }
    \\
    & =
    \frac{t(H,W)^{p_1}}{\prod_{i=1}^{n-1} t(H_{C_i},W)^{p_i-p_{i+1}}},
  \end{align*}
  where $p_{n+1}\df 1$ (note that the $n$th term of the final product is omitted because $t(H_{C_n},W)=1$ as
  $C_n=\varnothing$).

  Recall now that all $W_i$ except possibly for $W_n$ are left-regular and non-zero and since $n\in
  C_1\cap\cdots\cap C_{n-1}$, all $W_j$ with $j\in C_H\setminus C_i$ for some $i\in[n-1]$ are left-regular and
  non-zero, so by Lemma~\ref{lem:colorrestriction} (and recalling that $p_i\geq p_{i+1}$), we have
  \begin{equation}\label{eq:thWcolor}
    \begin{aligned}
      t(h,W)
      & \geq
      \frac{t(H,W)^{p_1}}{\prod_{i=1}^{n-1} t(H_{C_i},W)^{p_i-p_{i+1}}}
      \\
      & \geq
      t(H,W)^{p_1}\cdot
      \prod_{i=1}^{n-1}\left(\frac{\prod_{j\in C_H\setminus C_i} t(\rho,W_j)^{e_j(H)}}{t(H,W)}\right)^{p_i-p_{i+1}}
      \\
      & =
      t(H,W)^{p_n}\cdot \prod_{i=1}^n t(\rho,W_i)^{e_i(H)\cdot (p_i - p_n)}
      \\
      & \geq
      t(\rho_H,W)^{p_n\cdot e(H)}\cdot\prod_{i=1}^n t(\rho,W_i)^{e_i(H)\cdot (p_i - p_n)}
    \end{aligned}
  \end{equation}
  where the last inequality follows since $H$ is color-Sidorenko by Lemma~\ref{lem:leftweakHolder->colorSid}.

  Note now that since $p_n\cdot e(H) + \sum_{i=1}^n e_i(H)\cdot (p_i - p_n) = e(h)$, by \Holder's Inequality,
  we have
  \begin{align*}
    t(\rho_h,W)
    & =
    \int_X \prod_{i=1}^n t(e_1,W_i)^{p_i\cdot e_i(H) / e(h)}\ d\mu(x)
    \\
    & =
    \int_X \left(\prod_{i=1}^n t(e_1,W_i)^{p_n\cdot e_i(H) / e(h)}\right)
    \cdot \prod_{i=1}^n t(e_1,W_i)^{e_i(H)\cdot (p_i - p_n) / e(h)}
    \ d\mu(x)
    \\
    & \leq
    \left(\int_X \prod_{i=1}^n t(e_1,W_i)^{e_i(H) / e(H)}\right)^{p_n\cdot e(H)/e(h)}
    \cdot
    \prod_{i=1}^n t(\rho,W_i)^{e_i(H)\cdot (p_i - p_n)/e(h)}
    \\
    & =
    t(\rho_H,W)^{p_n\cdot e(H)/e(h)}
    \cdot
    \prod_{i=1}^n t(\rho,W_i)^{e_i(H)\cdot (p_i - p_n)/e(h)}.
  \end{align*}
  Plugging the $e(h)$th power of this inequality in~\eqref{eq:thWcolor} then gives $t(h,W)\geq
  t(\rho_h,W)^{e(h)}$, that is, $h = H^{\vec{p}}$ is color-Sidorenko.
\end{proof}

We can finally prove Theorem~\ref{thm:leftweakHolderorbits} (restated below).

\leftweakHolderorbits*

\begin{proofof}{Theorem~\ref{thm:leftweakHolderorbits}}
  First, by~\cite[Theorem~2]{Sid91}, the class of strong Sidorenko bigraphs is closed under amalgamation with
  a single edge along a vertex; thus, by an inductive application of this result, it is sufficient to prove
  the case when all vertices on the right of $G$ have degree at least $2$, or equivalently, we have $d_G(U) =
  0$ whenever $\lvert U\rvert\leq 1$.
  
  Without loss of generality, let us assume $C_H = \im(c_H)$.

  Since $H$ is right-uniform and does not have isolated vertices, we can define a coloring
  $r_H\function{V_2(H)}{C_H}$ of the right side of $H$ by letting $r_H(w)$ be the color of any (equivalently, all)
  edges incident to $w$. Given a set $U\subseteq V_1(H)$, let us also define
  \begin{align*}
    D_G(U) & \df \{w\in V_2(G) \mid N_G(w) = U\}, &
    D_H(U) & \df \{w\in V_2(H) \mid N_H(w) = U\},
  \end{align*}
  so that $d_G(U)=\lvert D_G(U)\rvert$ and $d_H(U)=\lvert D_H(U)\rvert$. For each $i\in C_H$, let
  \begin{align*}
    \cU_i & \df \{U\subseteq V_1(H) \mid \exists w\in D_H(U), r_H(w) = i\}.
  \end{align*}
  Note that for $U\in\cU_i$, if $\Aut(H)\cdot U$ is the orbit of $U$ under the action of
  $\Aut(H)$, then color-edge-transitivity of $H$ implies that $\Aut(H)\cdot U=\cU_i$.

  \begin{claim}\label{clm:colornhdcount}
    For every $U_1,U_2\subseteq V_1(H)$ and every $i\in C_H$, if $D_H(U_1)\cap r_H^{-1}(i)$ and $D_H(U_2)\cap
    r_H^{-1}(i)$ are non-empty, then $\lvert U_1\rvert = \lvert U_2\rvert$ and $\lvert D_H(U_1)\cap
    r_H^{-1}(i)\rvert = \lvert D_H(U_2)\cap r_H^{-1}(i)\rvert$.
  \end{claim}

  \begin{proof}
    For $j\in[2]$, let $u_j\in U_j$ and let $w_j\in D_H(U_j)\cap r_H^{-1}(i)$ so that $c_H(u_j,w_j) =
    i$. Since $H$ is color-edge-transitive, there exists $\sigma\in\Aut(H)$ such that $\sigma(u_1) = u_2$ and
    $\sigma(w_1) = w_2$ and since $N_H(w_j) = U_j$ ($j\in[2]$), we must have $\sigma(U_1) = U_2$, hence
    $\lvert U_1\rvert = \lvert U_2\rvert$. In turn, since $\sigma(U_1) = U_2$, it also follows that $\lvert
    D_H(U_1)\cap r_H^{-1}(i)\rvert = \lvert D_H(U_2)\cap r_H^{-1}(i)\rvert$ as $\sigma$ is an automorphism.
  \end{proof}

  Let us start by proving the case in which for every $U\subseteq V_1(H)$, we have $\lvert r_H(D_H(U))\rvert\leq
  1$, that is, all vertices of $D_H(U)$ have the same color. Note that this hypothesis implies that the
  $\cU_i$ are pairwise disjoint. In fact, this along with the hypothesis~\ref{thm:leftweakHolderorbits:zero}
  of the theorem gives that each $w\in V_2(G)$ belongs to exactly one set of the form $N_G(U)$ for some
  $U\in\cU_i$ and some $i\in C_H$ (recall that the degree of $w$ is at least $2$). For each $i\in C_H$, define
  \begin{align*}
    d_i(H)
    & \df
    \frac{\lvert\cU_i\rvert}{\lvert\Aut(H)\rvert}\cdot\sum_{\sigma\in\Aut(H)} d_H(\sigma(U)),
    \\
    d_i(G)
    & \df
    \frac{\lvert\cU_i\rvert}{\lvert\Aut(H)\rvert}\cdot\sum_{\sigma\in\Aut(H)} d_G(\sigma(U)),
    \\
    p_i
    & \df
    d_i(G)/d_i(H),
    \\
    m_i
    & \df
    \lvert U\rvert,
  \end{align*}
  where $U$ is any set in $\cU_i$. Claim~\ref{clm:colornhdcount} implies that the definitions above do not
  depend on the choice of $U$. Note also that hypothesis~\ref{thm:leftweakHolderorbits:geq1} of the theorem
  gives $p_i\geq 1$. Finally, define
  \begin{align*}
    \cV_i & \df \{w\in V_2(G) \mid \exists U\in\cU_i, N_G(w) = U\}
  \end{align*}
  so that $\lvert\cV_i\rvert = d_i(G)$.

  Fix a bigraphon $W\function{\Omega\times\Lambda}{\RR_+}$ and sequences $f=(f_v)_{v\in V_1(G)}$ and
  $g=(g_w)_{w\in V_2(G)}$ of bounded measurable functions $f_v\function{\Omega}{\RR_+}$ and
  $g_w\function{\Lambda}{\RR_+}$. For each $i\in C_H$, we define the function
  $\widehat{g}_i\function{\Lambda}{\RR_+}$ by
  \begin{align*}
    \widehat{g}_i(y) & = \prod_{w\in\cV_i} g_w(y)^{1/d_i(G)},
  \end{align*}
  that is, $\widehat{g}_i$ is the geometric average of the sequence $(g_w \mid w\in\cV_i)$.

  Note now that by renaming the variables, for every $\sigma\in\Aut(H)$, we have
  \begin{align*}
    t(G;f,g;W)
    & =
    \begin{multlined}[t]
      \int_{X^{V_1(G)}\times Y^{V_2(G)}}
      \prod_{v\in V_1(G)} f_v(x_{\sigma(v)}) \prod_{w\in V_2(G)} g_w(y_{\sigma(w)})
      \\
      \cdot\prod_{(v,w)\in E(G)} W(x_{\sigma(v)},y_{\sigma(w)})
      \ d(\mu\otimes\nu)(x,y)
    \end{multlined}
    \\
    & =
    \int_{X^{V_1(G)}}
    \prod_{v\in V_1(G)} f_v(x_{\sigma(v)})
    \cdot
    \prod_{i\in C_H} \prod_{U\in\cU_i} \prod_{w\in D_G(U)} t(K_{m_i,1}^L,W_w)(x_{\sigma(U)})
    \ d\mu(x),
  \end{align*}
  where $W_w(x,y) \df W(x,y) g_w(y)$ and the second equality follows from the fact that each $w\in V_2(G)$
  belongs to exactly one set of the form $N_G(U)$ for some $U\in\cU_i$ and some $i\in C_H$. Then \Holder's
  Inequality implies
  \begin{multline}\label{eq:symmetrization}
    t(G;f,g;W)
    \geq
    \int_{X^{V_1(G)}}
    \prod_{\sigma\in\Aut(H)}
    \biggl(
    \prod_{v\in V_1(G)} f_v(x_{\sigma(v)})
    \\
    \cdot
    \prod_{i\in C_H} \prod_{U\in\cU_i} \prod_{w\in D_G(U)} t(K_{m_i,1}^L,W_w)(x_{\sigma(U)})
    \biggr)^{1/\lvert\Aut(H)\rvert}
    \ d\mu(x).
  \end{multline}

  Let us analyze each of the terms under the integral above.

  For the part corresponding to the family of functions $f$, by Lemma~\ref{lem:leftweakHolder->colorSid}, we
  know that $H$ is left-color-regular and since it is also color-edge-transitive, it follows that it is
  left-vertex-transitive. Finally, since $V_1(G) = V_1(H)$, we get
  \begin{align}\label{eq:fsymm}
    \prod_{\sigma\in\Aut(H)}
    \prod_{v\in V_1(G)} f_v(x_{\sigma(v)})^{1/\lvert\Aut(H)\rvert}
    & =
    \prod_{v_1,v_2\in V_1(G)} f_{v_1}(x_{v_2})^{1/v_1(G)}.
  \end{align}

  The other part is more involved: fix $i\in C_H$ and note that
  \begin{multline*}
    \prod_{\sigma\in\Aut(H)}\prod_{U\in\cU_i}\prod_{w\in D_G(U)}
    t(K_{m_i,1}^L,W_w)(x_{\sigma(U)})^{1/\lvert\Aut(H)\rvert}
    \\
    =
    \prod_{\sigma\in\Aut(H)}\prod_{U\in\cU_i}\prod_{w\in D_G(\sigma(U))}
    t(K_{m_i,1}^L,W_w)(x_U)^{1/\lvert\Aut(H)\rvert}.
  \end{multline*}
  Since for each $U\in\cU_i$, we have $\Aut(H)\cdot U = \cU_i$, it follows that for any given
  $(U,w)\in\cU_i\times\cV_i$, the factor $t(K_{m_i,1}^L,W_w(x_U))^{1/\lvert\Aut(H)\rvert}$ appears
  exactly $\lvert\Aut(H)\rvert/\lvert\cU_i\rvert$ times: exactly when $\sigma(U) = N_G(w)$, so we
  deduce
  \begin{align}\label{eq:gsymm}
    \prod_{\sigma\in\Aut(H)}\prod_{U\in\cU_i}\prod_{w\in D_G(U)}
    t(K_{m_i,1}^L,W_w)(x_{\sigma(U)})^{1/\lvert\Aut(H)\rvert}
    & =
    \prod_{U\in\cU_i}\prod_{w\in\cV_i}
    t(K_{m_i,1}^L,W_w)(x_U)^{1/\lvert\cU_i\rvert}.
  \end{align}

  On the other hand, by \Holder's Inequality, for each $U\in\cU_i$, we have
  \begin{equation}\label{eq:gsymm2}
    \begin{aligned}
      \prod_{w\in\cV_i} t(K_{m_i,1}^L,W_w)(x_U)^{1/\lvert\cU_i\rvert}
      & =
      \prod_{w\in\cV_i}
      \left(\int_Y \prod_{u\in U} W(x_u,y)\cdot g_w(y)\ d\nu(y)\right)^{1/\lvert\cU_i\rvert}
      \\
      & \geq
      \left(\int_Y \prod_{u\in U} W(x_u,y)
      \cdot \prod_{w\in\cV_i} g_w(y)^{1/d_i(G)}\ d\nu(y)
      \right)^{d_i(G)/\lvert\cU_i\rvert}
      \\
      & =
      t(K_{m_i,1}^L,\widehat{W}_i)(x_U)^{d_i(G)/\lvert\cU_i\rvert},
    \end{aligned}
  \end{equation}
  where $\widehat{W}_i(x,y)\df W(x,y)\cdot\widehat{g}_i(y)^{1/m_i}$.

  Putting~\eqref{eq:symmetrization}, \eqref{eq:fsymm}, \eqref{eq:gsymm} and~\eqref{eq:gsymm2} together, we get
  \begin{align*}
    t(G;f,g;W)
    & \geq
    \int_{X^{V_1(G)}}
    \prod_{v_1,v_2\in V_1(G)}
    f_{v_1}(x_{v_2})^{1/v_1(G)}
    \prod_{i\in C_H} \prod_{U\in\cU_i}
    t(K_{m_i,1}^L,\widehat{W}_i)(x_U)^{d_i(G)/\lvert\cU_i\rvert}
    \ d\mu(x)
    \\
    & =
    \int_{X^{V_1(G)}}
    \prod_{i\in C_H} \prod_{U\in\cU_i}
    t(K_{m_i,1}^L,W'_i)(x_U)^{d_i(G)/\lvert\cU_i\rvert}
    \ d\mu(x)
    \\
    & =
    \int_{X^{V_1(G)}}
    \prod_{i\in C_H} \prod_{U\in\cU_i}
    t(K_{m_i,1}^L,W'_i)(x_U)^{p_i\cdot d_i(H)/\lvert\cU_i\rvert}
    \ d\mu(x)
    \\
    & =
    t(H^{\vec{p}},W'),
  \end{align*}
  where $W'_i(x,y)\df \widehat{W}_i(x,y)\cdot\prod_{v\in V_1(G)} f_v(x)^{1/e(G)}$, the first equality
  follows since for each $i\in C_H$, each $v\in V_1(G)$ belongs to exactly $\lvert\cU_i\rvert\cdot m_i/v_1(G)$
  sets $U\in\cU_i$ (as $H$ is left-vertex-transitive and $v_1(H)=v_1(G)$) and we have
  \begin{align*}
    \sum_{i\in C_H}
    m_i\cdot
    d_i(G)
    & =
    e(G),
  \end{align*}
  and the second equality follows since $p_i = d_i(G)/d_i(H)\geq 1$. Since $H^{\vec{p}}$ is
  color-Sidorenko by Lemma~\ref{lem:leftweakHoldercolorpower}, we conclude that
  \begin{align*}
    t(G;f,g;W)
    & \geq
    t(H^{\vec{p}},W')
    \geq
    t(\rho_{H^{\vec{p}}},W')^{e(H^{\vec{p}})}.
  \end{align*}

  Note now that
  \begin{align*}
    e_i(H^{\vec{p}}) & = p_i\cdot m_i\cdot d_i(H) = m_i\cdot d_i(G),\\
    e(H^{\vec{p}}) & = \sum_{i\in C_H} m_i\cdot d_i(G) = e(G).
  \end{align*}
  This means that
  \begin{align*}
    t(\rho_{H^{\vec{p}}},W')
    & =
    \int_{X\times Y}
    \prod_{v\in V_1(G)} f_v(x)^{1/e(G)} \prod_{i\in C_H}\widehat{g}_i(y)^{d_i(G)/e(G)}
    \cdot W(x,y)
    \ d(\mu\otimes\nu)(x,y)
    \\
    & =
    \int_{X\times Y}
    \prod_{v\in V_1(G)} f_v(x)^{1/e(G)} \prod_{i\in C_H}\prod_{w\in\cV_i} g_w(y)^{1/e(G)}
    \cdot W(x,y)
    \ d(\mu\otimes\nu)(x,y)
    \\
    & =
    t\left(\rho;\prod_{v\in V_1(G)} f_v^{1/e(G)},\prod_{w\in V_2(G)} g_w^{1/e(G)};W\right).
  \end{align*}
  Therefore $G$ is a strong Sidorenko bigraph.

  \medskip

  We will now show how the general case reduces to the previous case. More specifically, we will show that if
  $G$ satisfies the hypotheses of the theorem for a general $H$, then it also satisfies the same hypotheses
  for some $H'$ satisfying $\lvert r_{H'}(D_{H'}(U))\rvert\leq 1$ for every $U\subseteq V_1(H')$.

  Given a general $H$ let $C\subseteq C_H$ be a minimal set of colors such that if $U\subseteq V_1(H)$
  satisfies $d_H(U)\geq 1$, then there exists $w\in D_H(U)$ with $r_H(w)\in C$. Let then $H'$ be the colored
  bigraph obtained from $H_C$ by removing all of its isolated vertices. Note that only vertices in $V_2(H)$
  can be removed in this procedure, so $V_1(H')=V_1(H)=V_1(G)$. It is also clear that $H'$ is right-uniform
  and by Remark~\ref{rmk:leftweakHoldersubcolor}, it follows that $H_C$ is left-weakly \Holder, which implies
  that $H'$ is also left-weakly \Holder.

  Note further that if $\sigma\in\Aut(H)$, then $\sigma\rest_{V(H')}\in\Aut(H')$, which in particular implies
  that $H'$ is color-edge transitive.

  Conversely, we claim that every element of $\Aut(H')$ is of the form $\sigma\rest_{V(H')}$ for some
  $\sigma\in\Aut(H)$. Indeed, to extend $\sigma\in\Aut(H')$ to an automorphism of $H$, we note that
  Claim~\ref{clm:colornhdcount} implies that for each $i\in C_H$ and each $U_1,U_2\in\cU_i$, there exists a
  bijection $\theta_{U_1,U_2}^i$ between $D_H(U_1)\cap r_H^{-1}(i)$ and $D_H(U_2)\cap r_H^{-1}(i)$, so
  defining
  \begin{align*}
    \sigma(w) & \df \theta_{N_H(w),\sigma(N_H(w))}^i(w)
  \end{align*}
  for every $w\in V_2(H)\setminus V_2(H')$ with $r_H(w) = i$ gives an extension of $\sigma$ to an automorphism
  of $H$.

  In particular, this means that the orbits of the actions of $\Aut(H)$ and $\Aut(H')$ on $V_1(H)=V_1(H')$ are
  the same.

  We claim that for every $U\subseteq V_1(H')$, we have $\lvert r_{H'}(D_{H'}(U))\rvert\leq
  1$. Suppose not and let $U_0\subseteq V_1(H')$ be such that $D_{H'}(U_0)\cap r_{H'}^{-1}(i_1)\cap
  r_{H'}^{-1}(i_2)$ is non-empty for some $i_1\neq i_2$. Note that we must necessarily have
  $i_1,i_2\in C$, which along with Claim~\ref{clm:colornhdcount} implies that for every $U\subseteq
  V_1(H')$, we have
  \begin{align}\label{eq:i1i2}
    D_H(U)\cap r_H^{-1}(i_1)\neq\varnothing
    & \iff
    D_H(U)\cap r_H^{-1}(i_2)\neq\varnothing.
  \end{align}
  Let us now show that $C\setminus\{i_2\}$ contradicts the minimality
  of $C$. From the choice of $C$, we know that if $U\subseteq V_1(H)$ is such that $d_H(U)\geq 1$,
  then there exists $w\in D_H(U)$ with $w\in C$, but from~\eqref{eq:i1i2}, we conclude that there
  exists $w'\in D_H(U)$ with $w'\in C\setminus\{i_2\}$, thus $C\setminus\{i_2\}$ also satisfies the
  same property defining $C$, contradicting its minimality. This concludes the proof of the
  claim. Thus, the previous case of the theorem can be applied to $H'$.

  It remains to show that the hypotheses~\ref{thm:leftweakHolderorbits:zero}
  and~\ref{thm:leftweakHolderorbits:geq1} of the theorem for $G$ and $H$ imply that the same
  hypotheses hold for $G$ and $H'$. But indeed, from the definition of $H'$ and since the orbits of
  the actions of $\Aut(H)$ and $\Aut(H')$ on $V_1(H)=V_1(H')=V_1(G)$ are the same, we get
  \begin{align*}
    d_H(U) & \geq d_{H'}(U), &
    d_H(U) = 0 & \iff d_{H'}(U) = 0
  \end{align*}
  for every $U\subseteq V_1(H)$, so the hypotheses~\ref{thm:leftweakHolderorbits:zero}
  and~\ref{thm:leftweakHolderorbits:geq1} of the theorem for $G$ and $H$ imply that the same hypotheses hold
  for $G$ and $H'$.
\end{proofof}

\section{Conclusion and open problems}
\label{sec:conc}

In this paper, we have shown how left-sided analogues of the concepts of reflection bigraphs and
cut-percolating bigraphs can be used to obtain induced-Sidorenko bigraphs. We also showed that the
left-sided version of the weakly \Holder\ property, along with color-edge transitivity and
right-uniformity, also follow from left-reflection and can be used along with a standard
symmetrization technique to obtain the strong Sidorenko property.

\medskip

In the proof of Theorem~\ref{thm:leftcutperc->indSid}, we exploited heavily the fact that weak
domination (hence also the induced-Sidorenko property) only uses target bigraphons that are
biregular. On the other hand, Szegedy~\cite{Sze15b} showed that there is no loss in generality in
studying Sidorenko's Conjecture only when the target bigraphs are both edge-vertex-transitive, so it
is natural to ask what advantages this stronger assumption can yield.

As mentioned in the introduction, it was shown in~\cite{DGHRR18} that the weak norming property can
be equivalently restated as an extremal property called step Sidorenko property, which was studied
in~\cite{KMPW19} and implicitly in~\cite[\S14.2]{Lov12}. Since both the induced-Sidorenko property
and the left-weak \Holder-property are weaker analogues of the weak norming property, it is natural
to ask if there are similar characterizations of them in terms of an extremal property.

In~\cite[Conjectures~6.1 and~6.2]{CL17}, Conlon--Lee conjectured that a bigraph is weakly norming if
and only if it is edge-transitive under its cut-involution group. A similar but different conjecture
is that a bigraph is cut-percolating if and only if it is edge-transitive under the group generated
by cut-involutions coming from folds (see Remark~\ref{rmk:cutinv}). Analogously, one can conjecture
that a bigraph is left-cut-percolating if and only if it is left-vertex-transitive under the group
generated by cut-involutions coming from folds. Since there is a mismatch between the fact that the
left-cut-percolation property is defined for bigraphs and left-weakly \Holder\ property is defined
for \emph{colored} bigraphs, we believe that a decent analogous conjecture relating the two might
require some hypotheses how the coloring relates to folds
(cf.~Theorem~\ref{thm:leftcutperc->leftweakHolder}).

\section*{Acknowledgment}

I am grateful to Alexander Sidorenko for bringing to my attention the references~\cite{Sid20,LS21,LS22}.

\appendix

\section{Reflective tree decompositions}
\label{sec:reftree}

This section contains the definition of reflective tree decompositions (and the definition of
$2$-cores required by it) and the associated result from~\cite{CR21}.

\begin{definition}[$2$-cores]
  For a connected bigraph $G$, the \emph{$2$-core} of $G$ is the maximal connected subgraph $C_2(G)$
  in which all vertices have degree at least $2$. Alternatively, $C_2(G)$ can be obtained from $G$
  by progressively removing, in an arbitrary order, vertices of degree less than $2$ until no such
  vertices remain.

  For a flag $F = (G,\theta)$ with $G$ connected, the \emph{$2$-core} of $F$ is the flag of the form
  $F' = (G',\theta)$, where $G'$ is the maximal subgraph in which all vertices that are \emph{not}
  in $\im(\theta)$ have degree at least two; this can of course be obtained by progressively
  removing vertices of degree less than $2$ that are not in $\im(\theta)$ until no such vertices
  remain.
\end{definition}

\begin{remark}\label{rmk:2core}
  Since in the definition of weak domination the target bigraphons are biregular, it follows that
  $G$ weakly dominates $H$ if and only if $C_2(G)$ weakly dominates $C_2(H)$.
\end{remark}

\begin{definition}[Reflective tree decompositions]
  Given a connected non-trivial bigraph $G$, a \emph{reflective tree decomposition} of $G$ is a tree
  $T$ such that
  \begin{enumerate}
  \item We have $V(T)\subseteq 2^{V(G)}$ and $V(G) = \bigcup_{U\in V(T)} U$.
  \item For every $(v,w)\in E(G)$, there exists $U\in V(T)$ with $v,w\in U$.
  \item For every $U_1,U_2\in V(T)$ and every $U_3\in V(T)$ in the unique path from $U_1,U_2\in
    V(T)$ and every $U_3\in V(T)$ in the unique path from $U_1$ to $U_2$ in $T$, we have $U_1\cap
    U_2\subseteq U_3$.
  \item For every $\{U_1,U_2\}\in E(T)$, we have $C_2(F_{U_1 U_2})\cong C_2(F_{U_2 U_1})$, where
    $F_{U_i U_j}\df (G\rest_{U_i}, U_1\cap U_2)$ (we assume that each vertex of $U_1\cap U_2$
    receives the same label in $F_{U_1 U_2}$ as in $F_{U_2 U_1}$).
    \label{it:reflection}
  \end{enumerate}

  Condition~\ref{it:reflection} above implies that $C_2(G\rest_{U_1})\cong C_2(G\rest_{U_2})$ for
  every $U_1,U_2\in V(T)$; this common $2$-core bigraph is called the \emph{core} of the
  decomposition.
\end{definition}

\begin{theorem}[\protect{\cite[Theorem~3.5]{CR21}}]
  If $T$ is a reflective tree decomposition of a non-trivial connected bigraph $G$ whose core $H$
  weakly dominates $G\rest_{U_1\cap U_2}$ for every $\{U_1,U_2\}\in E(T)$, then $G$ weakly dominates
  $H$.

  In particular, if $H$ is a Sidorenko bigraph, then $G$ is also a Sidorenko bigraph.
\end{theorem}

In the theorem above, by Remark~\ref{rmk:2core}, we have that $H$ weakly dominates $G\rest_{U_1\cap
  U_2}$ if and only if $H$ weakly dominates $C_2(G\rest_{U_1\cap U_2})$ and since the latter is a
induced subgraph of $H$, we get the following corollary when $H$ is an induced-Sidorenko bigraph.

\begin{corollary}[\cite{CR21}]
  If $G$ is a non-trivial connected bigraph with a reflective tree decomposition with an
  induced-Sidorenko core $H$, then $G$ weakly dominates $H$ and $G$ is a Sidorenko bigraph.
\end{corollary}

\section{Strong Sidorenko bigraphs}
\label{sec:strongSid}

In~\cite[\S 2]{Sid91}, Sidorenko defined the class $\cF$ as the class of bigraphs $G$ such that
\begin{multline}\label{eq:SidcF}
  \int_{X^{V_1(G)}\times Y^{V_2(G)}}
  \prod_{v\in V_1(G)} f_v(x_v)\cdot
  \prod_{w\in V_2(G)} g_w(y_w)\cdot
  \prod_{(v,w)\in E(G)} W(x_v,y_w)
  \ d(\mu\otimes\nu)(x,y)
  \\
  \cdot
  \left(\int_X F(x)\ d\mu(x)\right)^{e(G)-v_1(G)}
  \cdot
  \left(\int_Y G(y)\ d\nu(y)\right)^{e(G)-v_2(G)}
  \\
  \geq
  \left(
  \int_{X\times Y}
  \left(
  F(x)^{e(G)-v_1(G)}
  \cdot
  G(x)^{e(G)-v_2(G)}
  \cdot
  \prod_{v\in V_1(G)} f_v(x)
  \cdot
  \prod_{w\in V_2(G)} g_w(y)
  \right)^{1/e(G)}
  W(x,y)
  \ d(\mu\otimes\nu)
  \right)^{e(G)}
\end{multline}
for all finite measure spaces $(X,\mu)$ and $(Y,\nu)$ and all bounded measurable functions
$F,f_v\function{X}{\RR_+}$ ($v\in V_1(G)$), $G,g_w\function{Y}{\RR_+}$ ($w\in V_2(G)$) and
$W\function{X\times Y}{\RR_+}$. In fact, Sidorenko also required $e(G)\geq\max\{v_1(G),v_2(G)\}$ but
this condition follows from~\eqref{eq:SidcF} (see Remark~\ref{rmk:strongSid->Sid}).

It is obvious that any bigraph $G$ in the class $\cF$ is a strong Sidorenko bigraph in the sense of
Definition~\ref{def:strongSidorenko} by restricting to probability spaces and setting $F$ and $G$ to
be identically $1$.

For the other direction, suppose $G$ is a strong Sidorenko bigraph and $X,\mu,Y,\nu,F,f_v,G,g_w,W$
are as in~\eqref{eq:SidcF}. Replacing all functions $h$ by $h + \epsilon$ and using the Dominated
Convergence Theorem letting $\epsilon\to 0$, it is sufficient to show the case when all functions
are strictly positive. Define then the probability measures $\mu'$ and $\nu'$ by
\begin{align*}
  d\mu'(x) & \df \frac{F(x)}{M}\ d\mu(x), &
  d\nu'(y) & \df \frac{G(x)}{N}\ d\nu(y),
\end{align*}
where $M\df\int_X F\ d\mu$ and $N\df\int_Y G\ d\nu$ and let
\begin{align*}
  f'_v & \df \frac{f_v}{F}, & g'_w & \df \frac{g_w}{G}
\end{align*}
so that interpreting $W$ as a bigraphon over $(X,\mu')$ and $(Y,\nu')$, the left-hand side
of~\eqref{eq:SidcF} is written as
\begin{align*}
  t(G;f',g';W)\cdot (M\cdot N)^{e(G)}
\end{align*}
and the right-hand side of the same equation is written as
\begin{align*}
  \left(t\left(\rho;\prod_{v\in V_1(G)} (f'_v)^{1/e(G)},\prod_{w\in V_2(G)} (g'_w)^{1/e(G)};W\right)
  \cdot M\cdot N\right)^{e(G)}
\end{align*}
and thus~\eqref{eq:SidcF} follows from the fact that $G$ is a strong Sidorenko bigraph.

\bibliographystyle{alpha}
\bibliography{refs}

\end{document}